\documentclass[11pt]{article}

\usepackage{epsfig}
\usepackage{graphicx}
\usepackage{amsbsy}
\usepackage{amsmath}
\usepackage{amsfonts}
\usepackage{amssymb}
\usepackage{textcomp}
\usepackage{hyperref}
\usepackage{aliascnt}

\newcommand{\mcm}[3]{\newcommand{#1}[#2]{{\ensuremath{#3}}}} 

\mcm{\tuple}{1}{\langle #1 \rangle}
\mcm{\name}{1}{\ulcorner #1 \urcorner}
\mcm{\Nbb}{0}{\mathbb{N}}
\mcm{\Zbb}{0}{\mathbb{Z}}
\mcm{\Rbb}{0}{\mathbb{R}}
\mcm{\Cbb}{0}{\mathbb{C}}
\mcm{\Qbb}{0}{\mathbb{Q}}
\mcm{\Acal}{0}{\cal A}
\mcm{\Bcal}{0}{\cal B}
\mcm{\Ccal}{0}{\cal C}
\mcm{\Dcal}{0}{\cal D}
\mcm{\Ecal}{0}{\cal E}
\mcm{\Fcal}{0}{\cal F}
\mcm{\Gcal}{0}{\cal G}
\mcm{\Hcal}{0}{\cal H}
\mcm{\Ical}{0}{\cal I}
\mcm{\Jcal}{0}{\cal J}
\mcm{\Kcal}{0}{\cal K}
\mcm{\Lcal}{0}{\cal L}
\mcm{\Mcal}{0}{\cal M}
\mcm{\Ncal}{0}{\cal N}
\mcm{\Ocal}{0}{{\cal O}}
\mcm{\Pcal}{0}{{\cal P}}
\mcm{\Qcal}{0}{{\cal Q}}
\mcm{\Rcal}{0}{{\cal R}}
\mcm{\Scal}{0}{{\cal S}}
\mcm{\Tcal}{0}{{\cal T}}
\mcm{\Ucal}{0}{{\cal U}}
\mcm{\Vcal}{0}{{\cal V}}
\mcm{\Wcal}{0}{{\cal W}}
\mcm{\Xcal}{0}{{\cal X}}
\mcm{\Ycal}{0}{{\cal Y}}
\mcm{\Mfrak}{0}{\mathfrak M}

\mcm{\restric}{0}{\upharpoonright}
\mcm{\upset}{0}{\uparrow}
\mcm{\onto}{0}{\twoheadrightarrow}
\mcm{\smallNbb}{0}{{\small \mathbb{N}}}
\DeclareMathOperator{\preop}{op}
\mcm{\op}{0}{^{\preop}}

\newcommand{\se}{\subseteq}
{\begin{array}{c}
\setlength{\unitlength}{1em}}%
{\end{array}}

\usepackage{amsthm}

\newcommand{\theoremize}[2]{\newaliascnt{#1}{thm} \newtheorem{#1}[#1]{#2} \aliascntresetthe{#1}}

\theoremstyle{plain}
\newtheorem{thm}{Theorem}[section]
\theoremize{lem}{Lemma}
\theoremize{skolem}{Skolem}
\theoremize{fact}{Fact}
\theoremize{sublem}{Sublemma}
\theoremize{claim}{Claim}
\theoremize{obs}{Observation}
\theoremize{prop}{Proposition}
\theoremize{cor}{Corollary}
\theoremize{que}{Question}
\theoremize{oque}{Open Question}
\theoremize{con}{Conjecture}

\theoremstyle{definition}
\theoremize{dfn}{Definition}
\theoremize{rem}{Remark}
\theoremize{eg}{Example}
\theoremize{exercise}{Exercise}
\theoremstyle{plain}

\usepackage{verbatim}
\usepackage{enumerate}
\usepackage[all]{xy}

\usepackage{subfig}

\title{\scshape All graphs have tree-decompositions displaying their topological ends}

\author{Johannes Carmesin
\medskip 
\\
  {University of Cambridge}
}
\date{}
\newcommand{\sm}{\setminus}

\newtheorem{theorem}{Theorem}

\begin{document}

\maketitle
\begin{abstract}
We show that every connected graph has a spanning tree that displays all its topological ends. This proves a 1964 conjecture of 
Halin in corrected form, and settles a problem of Diestel from 1992.
\end{abstract}

\section{Introduction}

In 1931, Freudenthal introduced a notion of \emph{ends} for second countable Hausdorff spaces 
\cite{Freudenthal31}, and in particular for locally finite graphs \cite{Freudenthal_graph_ends}. 
Independently, in 1964, Halin \cite{halin64} introduced a notion of \emph{ends} for graphs, 
taking his cue directly from Carath\'{e}odory's \emph{Primenden} of simply connected regions of the 
complex plane \cite{Caratheodory_1913}. 
For locally finite graphs these two notions of ends agree.

For graphs that are not locally finite, Freudenthal's topological definition still makes sense, and 
gave rise to the notion of \emph{topological ends} of arbitrary graphs \cite{Ends}. In general, this 
no longer agrees with Halin's notion of ends, although it does for trees. 

Halin \cite{halin64} conjectured that the end structure of every connected graph can be displayed by 
the ends of a suitable spanning tree of that graph. He proved this for countable graphs. Halin's 
conjecture was finally disproved in the 1990s by Seymour and Thomas \cite{ST:end-faithful}, and 
independently by Thomassen \cite{Thomassen:endfaithful}.

In this paper we shall prove Halin's conjecture in amended form, based on the topological notion of 
ends rather than Halin's own graph-theoretical notion. 
We shall obtain it as a corollary of the following theorem, which proves a conjecture of Diestel  
\cite{Diestel_end_spanningtree_survey} of 1992 (again, in amended form): 
\begin{theorem}\label{mainthm_morally}\label{undom_td}
Every graph has a tree-decomposition $(T,\Vcal)$ of finite adhesion such that the ends of $T$ define 
precisely the topological ends of $G$. 
\newline
\noindent \emph{See \autoref{prelims1} for definitions.}
\end{theorem}

\vspace{0.3 cm}

The tree-decompositions constructed for the proof of \autoref{mainthm_morally} have several further 
applications. In 
\cite{C:undom_td} we use them to answer the question to what extent the ends of a graph - now in 
Halin's sense - have a tree-like structure at all.
In \cite{C:cycle_matroids}, we apply \autoref{mainthm_morally} to show that the topological cycles 
of any graph together with its topological ends induce a matroid.
We remark that although the existence of a tree-decomposition as in \autoref{mainthm_morally} 
for an arbitrarily subset of the vertex-ends in place of the topological ends implies the existence 
of a suitable spanning tree in Halin's sense for that subset by \autoref{proof_more_general}, the 
converse is not true, see \autoref{sketch_t2}.

\vspace{0.3 cm}

This paper is organised as follows.
In \autoref{prelims1} we explain the problems of Diestel and Halin in detail, after having given 
some basic definitions. 
In \autoref{ex_sec} we continue with examples related to these problems.
\autoref{prelims2} only contains material that is relevant for \autoref{nested_sets} in which we 
prove that every graph has a nested set of separations distinguishing the vertex-ends efficiently.
In \autoref{sec:undom_td}, we use this theorem to prove \autoref{undom_td}. Then we deduce Halin's 
amended conjecture. Finally, \autoref{concluding} contains concluding remarks.

\section{Definitions}\label{prelims1}

Throughout, notation and terminology for graphs are that of~\cite{DiestelBook10}.
And  $G$ always denotes a graph.

A \emph{vertex-end} in a graph $G$ is an equivalence class of rays (one-way infinite paths), where 
two rays are equivalent if they cannot be separated in $G$ by removing finitely many vertices. Put 
another way, this equivalence relation is the transitive closure of the relation relating two rays 
if they intersect infinitely often. 

\begin{eg}
 The vertex-ends of rooted trees are (in bijection with) the rays starting at the root; of course 
vertex-ends do not depend on the choice of a root. 
\end{eg}

Let $X$ be a locally connected Hausdorff space. 
Given a subset $Y\se X$, we write $\overline Y$ for the closure of $Y$, and 
$F(Y) := \overline Y \cap \overline{X \setminus Y}$ for its frontier. In
order to define the topological ends of $X$, we consider infinite sequences
${U_1\supseteq U_2 \supseteq  . . .}$ of non-empty connected open subsets of $X$ such that each $F( U_i)$ is compact
and $\bigcap_{i\geq 1} \overline U_i=\emptyset$. 
We say that two such sequences ${U_1\supseteq U_2 \supseteq  . . .}$ and ${U_1'\supseteq U_2' \supseteq  . . .}$ are \emph{equivalent} if for every $i$ there 
is some $j$ with  $U_i\supseteq U_j'$.
This relation is transitive and symmetric \cite[Satz 2]{Freudenthal31}.
The equivalence classes of those sequences are the \emph{topological ends} of $X$ 
\cite{{Ends},{Freudenthal31},{HughesRanicki}}. 

For the simplicial complex of a graph $G$, Diestel and K\"uhn described the topological ends 
combinatorically: a vertex \emph{dominates} a vertex-end $\omega$ if for some (equivalently: every) 
ray $R$ belonging to $\omega$ there is an infinite fan of $v$-$R$-paths that are vertex-disjoint 
except at $v$. In \cite{Ends}, they proved that the  topological ends are given by the undominated 
vertex-ends. 
Hence in this paper, we take this as our definition of \emph{topological end of $G$}.

\begin{eg}
 For locally finite graphs the notions of vertex-ends and topological ends agree.
\end{eg}

\begin{eg}
For trees the notions of vertex-ends and topological ends agree. Hence we just call the vertex-ends 
of trees \emph{ends}. 
\end{eg}

\vspace{0.3 cm}

For us, a \emph{separation} is an (ordered) pair $(A,B)$ of vertex sets $A$ and $B$ such that no 
edge has an endvertex in $A\sm B$ and the other endvertex in $B\sm A$. The set $A\cap B$ is called 
the \emph{separator} of $(A,B)$. The size of the separator is the \emph{order} of $(A,B)$. The sets 
$A$ and $B$ are called the \emph{sides} of the separation. 
The \emph{reverse} of the separation $(A,B)$ is the separation $(B,A)$. 

Given two separations $(A,B)$ and $(C, D)$, we write $(A,B)\leq (C,D)$ if $A\se C$ and $D\se B$. 
These separations are \emph{nested} if $(A,B)\leq (C,D)$ or one of the other three possibilities 
obtained by replacing $(A,B)$ or $(C,D)$ by their reverse.
Formally, $(A,B)$ and $(C,D)$ are nested if $(A,B)\leq (C,D)$, $(B,A)\leq (C,D)$, $(A,B)\leq (D,C)$ 
or $(B,A)\leq (D,C)$. 
\begin{rem}
 Most separations of interest are `proper', see below. By \autoref{proper_nested}, proper 
separations $(A,B)$ and $(C, D)$ satisfy $(A,B)\leq (C,D)$ already if $A\se C$. In this sense our 
definition of nestedness corresponds to the notion of nestedness for sets. 
\end{rem}

A separation $(A,B)$ is \emph{proper} if every vertex in the separator $A\cap B$ has a neighbour in 
$A\sm B$ and $B\sm A$. 
\begin{obs}\label{proper_nested}
For proper separations $(A,B)$ and $(C,D)$ the following are equivalent.
\begin{enumerate}
 \item $(A,B)\leq(C,D)$;
 \item $A\se C$;
 \item $A\sm B\se C\sm D$.
\end{enumerate}
\end{obs}

\begin{proof}
 As $(A,B)$ is proper, the side $A$ is determined by the set $A\sm B$; indeed, it is  $A\sm B$ 
together with its neighbourhood. Conversely, also the side $A$ determines the set $A\sm B$: this 
set consists of those vertices of $A$ that have all their neighbours in $A$. 
So (2) and (3) are equivalent. 

Clearly (1) implies (2). Now conversely assume that $A\se C$. 
By the above it suffices to show that $D\sm C$ is included in $B\sm A$. 
In other words: $G\sm C$ is included in $G\sm A$. This follows from $A\se C$. 
\end{proof}

A vertex-end $\omega$ \emph{lives} in a side $B$ of a separation $(A,B)$ of finite order if the 
side $B$ includes a ray 
belonging to $\omega$. In this case $B$ includes a subray of every ray belonging to 
$\omega$, see \autoref{sep_ray}.
\begin{figure}
\begin{center}
   	  \includegraphics[height=3cm]{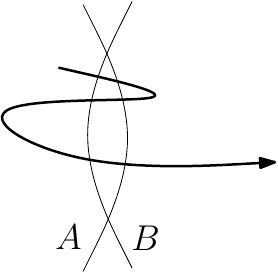}
   	  \caption{Every ray traverses the finite separator $A\cap B$ finitely often 
and then is eventually included in one of the sides $A$ or $B$.}\label{sep_ray}
\end{center}
   \end{figure}
A separation $(A,B)$ of finite order 
\emph{distinguishes} two vertex-ends $\omega$ and $\mu$ if
one of them lives in the side $A$ and the other lives in the side $B$. It distinguishes them 
\emph{efficiently} if $(A,B)$ has minimal order amongst all separations distinguishing $\omega$ and 
$\mu$.

A \emph{tree-decomposition} of a graph $G$ consists of a tree $T$ together with a family of 
subgraphs\footnote{We denote the vertex set of a graph $G$ by $V(G)$. } $(P_t|t\in V(T))$ of $G$ 
such that every vertex and edge of $G$ is in at least one of 
these subgraphs, and such that if $v$ is a vertex of both $P_t$ and $P_w$, then it is a vertex of 
each $P_u$, where $u$ lies on the $v$-$w$-path in $T$.
We call the subgraphs $P_t$ the \emph{parts} of the tree-decomposition.
The \emph{adhesion} of a tree-decomposition is finite if adjacent parts intersect only finitely.
Given an edge $tu$ of $T$, we denote by $T_t$ the subtree of $T-tu$ that contains $t$.
Given a directed edge $tu$ of $T$, the \emph{separation corresponding to $tu$} is the separation 
$(A_t,A_u)$, where $A_i$ is the union of all parts $P_x$ with $x\in T_i$ for $i=u,t$.

\vspace{0.3 cm}

In \cite{{BC:determinacy},{hp:transi}, {T:wqo_planar}}, tree-decompositions of finite adhesion are 
used to study the structure of infinite graphs.
In \cite[Problem 4.3]{Diestel_end_spanningtree_survey}, Diestel wanted to know whether every graph 
$G$ has a tree-decomposition $(T,P_t|t\in V(T))$ of finite adhesion that somehow encodes the 
structure of the graph with its ends. 

Let us be more precise. 
Given a vertex-end $\omega$, we take $O(\omega)$ to consist of those oriented edges $tu$ of $T$ such 
that $\omega$ lives in its corresponding separation.
Note that $O(\omega)$ contains precisely one of the two directions $tu$ and $ut$ of each edge of 
the tree.
Furthermore this orientation $O(\omega)$ of $T$ points towards a node of $T$ or to an end of $T$.
We say that $\omega$ \emph{lives} in the part for that node or that end, respectively.

A vertex-end $\omega$ is \emph{thin} if every set of vertex-disjoint rays belonging to $\omega$ is 
finite; otherwise $\omega$ is \emph{thick}.
Diestel asked whether every graph has a tree-decomposition $(T,P_t|t\in V(T))$ of finite adhesion 
such that different thick vertex-ends live in different parts and 
such that the ends of $T$ \emph{define precisely} the thin vertex-ends; here the ends of $T$ 
\emph{define precisely} a set $\Psi$ of vertex-ends of $G$ if
in every end of $T$ there lives a unique vertex-end and it is in $\Psi$ and conversely every 
vertex-end in $\Psi$ lives in some end of $T$, see \autoref{def_precisely}. 

\begin{figure}
\begin{center}
   	  \includegraphics[height=5cm]{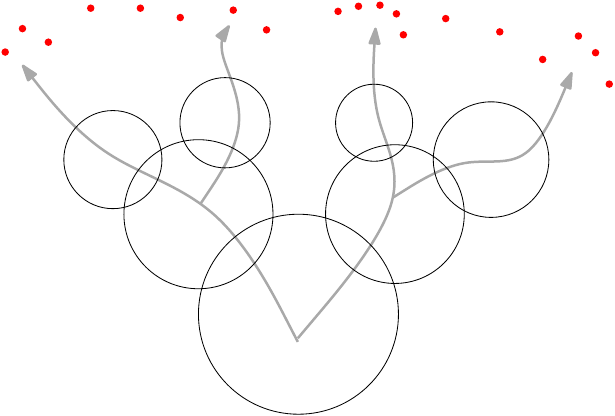}
   	  \caption{The ends of the decomposition tree - this tree is indicated in grey -  
define precisely the vertex-ends of the graph indicated by dots. Other vertex-ends 
living in parts are not drawn. }\label{def_precisely}
\end{center}
   \end{figure}

Unfortunately, that is not true; in \autoref{sketch_t2}, we construct a graph such that each of 
its 
tree-decompositions of finite adhesion has a part in which two (thick) vertex-ends live. 
In \autoref{thin_not_distinguished} we refine that construction by constructing a graph such that 
there live two thin vertex-ends in some part of every such a tree-decomposition.

Hence there remains the open question whether there is a natural subclass of the vertex-ends 
(similar 
to the class of thin vertex-ends) such that every graph has a tree-decomposition  of finite 
adhesion such that the ends of its decomposition tree define precisely the vertex-ends in that 
subclass. Another question that arises in this context is: what is the largest possible natural 
class of vertex-ends such that every graph has a 
tree-decomposition distinguishing the vertex-ends in that class?
\autoref{undom_td} above answers the first question affirmatively.
In \autoref{concluding}, we show how \autoref{undom_td} can be used to obtain a satisfying answer 
to the second question. 

It is impossible to construct a tree-decomposition as in \autoref{undom_td} with the additional property that for any two topological ends $\omega$ and $\mu$, 
there is a separation corresponding to an edge of the tree that separates $\omega$ and $\mu$ efficiently, see \autoref{not_efficiently}.

A recent development in the theory of infinite graphs seeks to extend theorems about finite graphs and their cycles to infinite graphs
and the topological circles formed with their ends, see for example
\cite{{BergerBruhnDeg},   {Degree},{cyclesI},{CyclesII}, {Georgakopoulos2009}, {Arboricity}}, and 
\cite{RDsBanffSurvey} for a survey.
We expect that \autoref{undom_td} has further applications in this direction aside from the one 
mentioned in the introduction.

\vspace{0.3 cm}

A rooted spanning tree $T$ of a graph $G$ is \emph{end-faithful} for a set $\Psi$ of vertex-ends if 
each vertex-end $\omega\in \Psi$ is uniquely represented by $T$ in the sense that
$T$ contains a unique ray belonging to $\omega$ and starting at the root.
For example, every normal spanning tree is end-faithful for all vertex-ends.
Halin conjectured that every connected graph has an end-faithful tree for all vertex-ends. 
At the end of \autoref{sec:undom_td}, we show that \autoref{undom_td} implies the following nontrivial weakening of this disproved conjecture:

\begin{cor}\label{Halin_type_intro}
Every connected graph has an end-faithful spanning tree for the topological ends.
\end{cor}

One might ask whether it is possible to construct an end-faithful spanning tree for the topological 
ends with the additional property that it does not include any ray to any other vertex-end. However, 
this is not possible in general. Indeed, Seymour and Thomas constructed a graph $G$ with no 
topological end that does not have a rayless spanning tree \cite{ST:end-faithful}.

\section{Example section}\label{ex_sec}

\begin{eg}\label{sketch_t2}
In this example we give two constructions of graphs that have no tree-decompositions of  finite 
adhesion that distinguish all vertex-ends. These 
constructions motivate the construction of \autoref{thin_not_distinguished}, where we can construct 
such a graph not only for the class of vertex-ends but for the finer class of thin vertex-ends. 

The simplest example of a graph with no such tree-decomposition for the vertex-ends known to the 
author is the (infinite) binary tree with tops, see \autoref{T2top}; this graph is obtained from 
the binary tree $T_2$ by adding one new vertex for every ray starting at the root. 
This new vertex is adjacent to 
all vertices on that ray. We call these new vertices the \emph{tops}. 
\begin{figure}
\begin{center}
   	  \includegraphics[height=4cm]{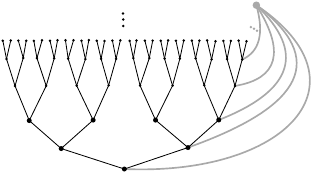}
   	  \caption{The binary tree is indicated in black. In grey we indicated the addition of a 
top along the right most ray.}\label{T2top}
\end{center}
   \end{figure}
We omit the proof that this graph has no tree-decompositions of  finite 
adhesion that distinguishes all vertex-ends\footnote{A proof can be found in an earlier version of 
this paper \cite{C:undom_td_v4}}.

A slightly more complicated example is obtained from the regular tree $T_\omega$ with countably 
infinite degree by adding fat tops; here adding \emph{fat tops} means that at each ray 
of $T_\omega$ starting at the root, we attach uncountably many, say $\aleph_1$, tops (that is new 
vertices adjacent to all vertices on the ray).

We sketch the proof that $T_\omega$ with fat tops has no tree-decompositions of 
 finite 
adhesion that distinguishes all vertex-ends. First one checks that the vertex-ends of 
$T_\omega$ with fat tops are the ends of $T_\omega$ (this proof is similar to 
\autoref{are_the_same_eg} below). The vertex-ends of $T_\omega$ with fat tops, 
however, are \emph{fat}, that is, they are dominated by uncountably many vertices. The key 
observation is the following. 
\begin{lem}\label{fat}
Let $H$ be any graph with a tree-decomposition $(T,\Vcal)$ of finite adhesion. 
Then no fat vertex-end of $H$ lives in an end of $T$. 
\end{lem}
\begin{proof}
Vertex-ends living in ends $\mu$ of $T$ can only be 
dominated by those vertices that eventually are in the separators corresponding to the edges on 
some ray in $\mu$. Since the tree-decomposition has finite adhesion, there can only be countably 
many such vertices. So vertex-ends living in ends of the decomposition tree cannot be fat. 
\end{proof}

In the final step one assumes that some 
tree-decomposition of finite adhesion distinguishes all vertex-ends. Since the graph $T_\omega$ is 
countable, it can 
only have countably many separators. A finite separator of $T_\omega$ with fat tops separates 
the same vertex-ends as their restriction to $T_\omega$ does. This essentially 
means\footnote{By contracting edges of the decomposition tree if necessary, we may assume that 
any two separations corresponding to edges of the decomposition tree distinguish different sets of 
vertex-ends. So no two such separations can have the same restriction to $T_\omega$. Hence we 
may assume that the decomposition tree has only have countably many edges.} that the 
decomposition tree has only countably many edges. So it can only have countably many nodes. Since 
there are uncountably many vertex-ends, two of them have to live in the same part as they cannot 
live in an end of the decomposition tree by \autoref{fat}. 

We remark that this proof also works for any graph obtained from $T_\omega$ by attaching 
some $\aleph_1$ fat tops at $T_\omega$. So there is a counterexample against the statement that 
every graph has a tree-decomposition of finite adhesion distinguishing its vertex-ends of 
cardinality $\aleph_1$ -- which is independent of the Continuum Hypothesis.

\end{eg}

\begin{eg}\label{thin_not_distinguished}
In this example we construct a graph $G$ such that each of its tree-decomposition of finite adhesion 
cannot distinguish all thin vertex-ends.

We start the construction with the regular tree $T_\omega$ of countably infinite degree. For each 
vertex of $T_\omega$, we add a ray through its neighbours in the next level.
Call the resulting graph $G'$, see 
\autoref{TomegaFAT+}.
\begin{figure}
\begin{center}
   	  \includegraphics[height=4cm]{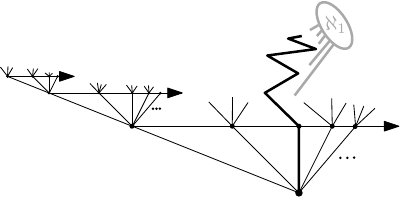}
   	  \caption{The graph $G'$ is indicated in black. We indicated in grey the addition of fat 
tops at the highlighted ray. We obtain the graph $G$ from the graph $G'$ by adding these 
fat tops at all rays starting at the root.}\label{TomegaFAT+}
\end{center}
   \end{figure}
The vertex-ends of $G'$ are those of $T_\omega$ together with one vertex-end for every newly added 
ray.

We obtain $G$ from $G'$ by adding for every ray of $T_\omega$ starting at the root a clique of 
uncountable cardinality 
$\aleph_1$ that is complete to that ray.

\begin{lem}\label{are_the_same_eg}
The vertex-ends of $G'$ are (in bijection with) the vertex-ends of $G$.
\end{lem}

\begin{proof}
Every ray of $G$ is equivalent to a ray of $G'$. Conversely any two vertex-ends of $G'$ can be 
separated by a path of $T_\omega$ starting at the root. This path still separates rays belonging 
to these vertex-ends of $G'$ in $G$. Hence $G$ and $G'$ have the same vertex-ends.  
 
\end{proof}

The thin vertex-ends of $G$ are those vertex-ends of $G'$ coming from newly added rays; indeed, 
if we remove the finite path of $T_\omega$ below such a newly added ray, all vertices on that ray 
become cut-vertices. All other  vertex-ends are each dominated by uncountably many vertices, that 
is, 
they are fat.

We use the vertices of $T_\omega$ to refer to the thin vertex-ends. More precisely, we say that 
the vertex-end \emph{sitting above} a vertex $v$ is the one to which the ray in the upward 
neighbourhood of $v$ belongs.

Suppose for a contradiction that the graph $G$ has a tree-decomposition $(T,P_t|t\in V(T))$ of 
finite 
adhesion that distinguishes all its thin vertex-ends.
First we show the following.

\begin{lem}\label{in_end}
 There is a ray $R$ of $T$ such that a fat vertex-end of $G$ lives in the 
end to which $R$ belongs.
\end{lem}

\begin{proof}
Our aim is to construct a sequence $(v_n|n\in \Nbb)$ of vertices that lie on a ray of the tree 
$T_\omega$ starting at the root together with a sequence $((A_n,B_n)|n\in \Nbb^*)$ of separations 
corresponding to edges of the decomposition tree such that $(A_n,B_n)\leq 
(A_{n+1},B_{n+1})$ and $v_n$ is contained in $B_n\sm A_n$. 
 
We start the construction by picking an arbitrary separation $(C,D)$ corresponding to an edge 
of the decomposition tree such that it distinguishes two thin vertex-ends. We pick for $v_0$ the 
root of the tree $T_\omega$. By replacing the separation $(C,D)$ by its reverse $(D,C)$ 
if necessary, we may assume that the thin vertex-end 
sitting above $v_0$ lives in the side $C$. We let $(A_1,B_1)=(C,D)$. 
Let $\mu$ be a thin vertex-end living in $B_1$ and let $u$ be the vertex of 
$T_\omega$ above which $\mu$ sits. The vertex $u$ must be contained in the side $B_1$ and have all 
but finitely many of its upward-neighbours in the side $B_1$. Since the separator $A_1\cap B_1$ is 
finite, the 
vertex $u$ has an upward-neighbour $v_1$ in the rooted tree $T_\omega$ that is contained in $B_1\sm 
A_1$. We let $P_1$ be the unique path included in the tree $T_\omega$ from the vertex $v_0$ to the 
vertex $v_1$. 
 
Now assume that we already constructed a path $P_n$ and a separation $(A_n,B_n)$ corresponding to 
an edge of the decomposition tree such that the last vertex $v_n$ of $P_n$ is contained in $B_n\sm 
A_n$ and is a vertex of $T_\omega$. Next we construct the path $P_{n+1}$ and the separation 
$(A_{n+1},B_{n+1})$.  
As the separator $A_n\cap B_n$ is finite, the vertex $v_n$ has two upward-neighbours $u$ and $u'$ 
in 
the rooted tree $T_\omega$ contained in $B_n\sm A_n$. By assumption there is a separation $(C,D)$ 
corresponding to an edge of the decomposition tree such that the thin vertex-ends sitting above $u$ 
and $u'$ are distinguished by $(C,D)$. 

\begin{sublem}\label{cases_nestedness_now}
The separation $(A_n,B_n)$ is $\leq$ to the separation $(C,D)$ or its reverse $(D,C)$.
\end{sublem}

\begin{proof}
This is a simple consequence of the fact that the separations $(C,D)$ and $(A_n,B_n)$ are nested as 
separations corresponding to edges of a decomposition tree of the same tree-decomposition. 

The sides $C$ and $B_n$ both contain all but finitely many vertices of every ray belonging to the 
vertex-end sitting above the vertex $u$. Hence the intersection $C\cap B_n$ is infinite. Similarly, 
we conclude that the intersection $D\cap B_n$ is infinite. As the separator $C\cap D$ is finite, 
the side $B_n$ cannot be included in one of the sides $C$ or $D$. Hence as the separations $(C,D)$ 
and $(A_n,B_n)$ are nested, it must be that the separation $(A_n,B_n)$ is $\leq$ to the separation 
$(C,D)$ or its reverse $(D,C)$.
\end{proof}

By replacing the separation $(C,D)$ by its reverse $(D,C)$ if necessary we may assume by 
\autoref{cases_nestedness_now} that $(A_n,B_n)\leq (C,D)$. We let $(A_{n+1}, 
B_{n+1})=(C,D)$. 
Since the separator $A_{n+1}\cap B_{n+1}$ is finite and the thin vertex-end sitting above $u'$ 
lives 
in $B_{n+1}$, the vertex $u'$ has an upward-neighbour $v_{n+1}$ in 
the rooted tree $T_\omega$ contained in $B_{n+1}\sm A_{n+1}$. We obtain the path $P_{n+1}$ from 
$P_n$ by adding the unique path included in $T_\omega$ from the vertex $v_n$ to the vertex 
$v_{n+1}$. 

This completes the construction of the paths $P_n$ and the separations $(A_n,B_n)$. Hence by 
recursion, there is a sequence $(v_n|n\in \Nbb)$ of vertices that lie on a ray $S$ of the tree 
$T_\omega$ starting at the root together with a sequence $((A_n,B_n)|n\in \Nbb^*)$ of separations 
corresponding to edges of the decomposition tree such that $(A_n,B_n)\leq 
(A_{n+1},B_{n+1})$ and $v_n$ is contained 
in $B_n\sm A_n$. The vertex-end $\mu$ to which the ray $S$ belongs is an end of the 
tree $T_\omega$; and thus is fat in the graph $G$. Since the ray $S$ contains infinitely many 
vertices of all sides $B_n$, its vertex-end $\mu$ lives in all sides $B_n$. The edges 
corresponding to the separations $(A_n,B_n)$ lie on a ray $R$ of the decomposition tree; and the 
vertex-end $\mu$ lives in the end of $R$. This completes the proof. 
\end{proof}

\autoref{in_end} contradicts \autoref{fat}. This is the desired contradiction. Hence $G$ has no 
tree-decomposition of finite adhesion that distinguishes all its thin vertex-ends.
\end{eg}

\begin{eg}\label{not_efficiently}
In this example, we construct a graph $G$ such that for any of its tree-decompositions $(T,P_t| 
t\in 
V(T))$ there are two topological ends such that no separation 
corresponding to an edge of $T$ distinguishes them efficiently\footnote{Topological ends are 
examples of vertex-ends. In this sense the term `distinguishes efficiently' is defined.}.

We start the construction with the (cartesian) product\footnote{Given two 
graphs $G$ and $H$, by $G\times H$, we denote the graph with vertex set $V(G)\times V(H)$ 
where we join two vertices $(g,h)$ and $(g',h')$ by an edge if both $g=g'$ and $hh'\in E(G)$ or 
both 
$h=h'$ and $gg'\in E(G)$.} $W$ of a 
ray with the path of five vertices, see \autoref{not_eff_fig}.    \begin{figure} [htpb]   
\begin{center}
   	  \includegraphics[height=4cm]{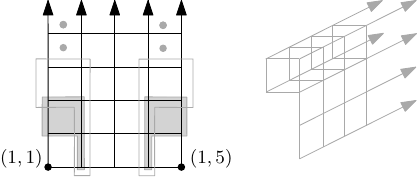}
   	  \caption{The construction of the graph $G$. It is obtained from the graph depicted on the 
left by attaching on each set of vertices surrounded by a `P'-shaped box a graph like 
the one on 
the right of the appropriate size.}\label{not_eff_fig}
\end{center}
   \end{figure}
   By $P[n]$, we denote a graph that has the shape of a `P'. More precisely, it is obtained from  
a path of $n$ vertices by adding an edge such that one endvertex of the edge is joined to the last 
vertex of the path and the other endvertex to the second but last. By $H_n$ we denote the product 
of a ray with the graph $P[n]$. 
We obtain $G$ from $W$ by for each $n\geq 3$ attaching two copies of $H_n$ as follows. We 
attach these new graphs $H_n$ on copies of $P[n]$. The first copy is that containing 
the initial path of the ray of length $n$ times the second vertex of the five-path together with 
the edge whose endvertices are the $n$-th and $(n-1)$-st vertex of the ray times the first vertex 
of the five-path. The second copy is that containing 
the initial path of the ray of length $n$ times the forth vertex of the five-path together with 
the 
edge whose endvertices are the $n$-th and $(n-1)$-st vertex of the ray times the fifth vertex of 
the five-path. In \autoref{not_eff_fig} 
these 
attachment sets are surrounded by grey `P'-shaped boxes. This completes the construction of $G$. 

The vertex-ends of the attached graphs $H_n$ are clearly topological. 
The graph $G$ has the property that although we attach the graphs $H_n$ at a copy of 
$P[n]$, the 
vertex-end of 
a new graph $H_n$ can be separated from the vertex-end of the other copy of $H_n$ by a separator 
properly contained in the attachment set $P[n]$; namely just those vertices in the attachment set 
that in $W$ have 
a neighbourhood in the infinite component of $W$ without the attachment set. The set of these 
vertices has the shape of 
an `$L$' turned around and consists of $n+1$ vertices. We denote these separators by $S_n^1$ and 
$S_n^2$, depending on whether they are contained in the first or second attachment set 
$P[n]$, respectively.

It is straightforward to check that any separation separating the two vertex-ends of the two 
attached 
copies of $H_n$ efficiently has the separating set $S_n^1$ or $S_n^2$.

Suppose for a contradiction that $G$ has a tree-decomposition $(T,P_t|t\in V(T))$ that separates 
any two topological ends efficiently. Then infinitely many of its separations must 
have separating sets of the form $S_n^1$ or $S_n^2$. By symmetry we may assume that there are 
infinitely many of the form $S_n^1$. 

By $(1,1)$ we denote the vertex of $G$ that is the product of the first vertex of the five-path and 
first 
vertex of the ray, see \autoref{not_eff_fig}. Similarly, by $(1,5)$ we denote the vertex of $G$ 
that is the product of the 
last vertex of five-path and first 
vertex of the ray.  Let $P_a$ be a part of the tree-decomposition that contains $(1,1)$ and 
similarly let $P_b$ be a part of the tree-decomposition that contains $(1,5)$. 
The edges corresponding to the separations with separators of the form $S_n^1$ 
separate in $T$ the vertex $a$ from the vertex $b$; that is, they lie on the unique $a$-$b$-path. 
Since this path is finite, we derive the desired contradiction.
Thus $G$ has  no tree-decomposition $(T,P_t|t\in V(T))$ such that for any two 
topological ends there is a separation 
corresponding to an edge of $T$ distinguishes them efficiently.

We remark that all topological ends of $G$ are thin vertex-ends and so this construction also shows 
that thin vertex-ends cannot always be distinguished efficiently.
\end{eg}

\section{Separations and tangles}\label{prelims2}

In this section, we define tangles and related concepts and prove some intermediate lemmas that we 
will apply in \autoref{nested_sets}. 

\subsection{Tangles}

Tangles are a central concept in Graph Minor Theory that describe highly connected substructures 
of a graph such as complete subgraphs or grid minors. 
They do not explicitly describe these substructures. Instead, for every low order separation they 
point towards a side, where that substructure `lives'. This side is called the \emph{big} 
side and the other side of the separation is the \emph{small} side. These assignments have to 
satisfy certain rules such as sides including big sides are big. 

Formally, a \emph{tangle of order $k+1$} assigns to each separation of order\footnote{We follow 
the convention that we allow $k+1$ to be infinite. In that case we just replace `of order at 
most 
$k$' by `of finite order' in the above definition. } at most $k$ a 
big side. The other side is called small. These assignments satisfy the following properties:
\begin{enumerate}
 \item three small sides $A_1$, $A_2$, $A_3$ cannot cover all edges, \newline in 
formulas: $G\neq G[A_1]\cup G[A_2]\cup G[A_3]$; 
\item if $X$ is a set of at most $k$ vertices, 
there is a component $C$ of $G-X$ such that 
$C\cup X$ is the big side of the separation $(C\cup X, G\sm C)$.
\end{enumerate}

From the first property it follows that if $(A,B)$ is a separation of order at most $k$ and $A\se 
B$, then $A$ is small and $B$ is big in any tangle of order $k+1$.
In particular, the empty set $\emptyset$ is the  small side of $(\emptyset, G)$.
Furthermore every separation of order at most $k$ has  
precisely one big side in a tangle of order $k+1$ by the first property. 
And a side including a big side cannot be small. Thus if a side is a big side of some 
separation, it must be the big side of any separation it is a side of. Thus we shall say 
things like 
`$A$ is big' without specifying a separation $(A,B)$ of which $A$ is the big side. 

\begin{rem}
 In the standard definition of tangles for finite graphs (or more generally for locally finite 
graphs), the second property is omitted. The reason is that for finite graphs there is a simple 
well-known argument that it follows from the first. This argument relies on an induction on the 
number of components of $G-X$ and this implication no longer holds for quite simple infinite graphs 
like the infinite star (in fact without the second conditions non-principle ultra-filters on the 
leaves of the infinite star would give rise to a tangle of infinite order). It is not the scope 
of this 
paper to analyse such objects.\footnote{Tangles without this second condition are studied in 
\cite{diestel_ultrafilter-tangles} by Diestel.} Hence we require this second condition. 

We refer to this second condition as the \emph{component property}.
\end{rem}

In this paper 
we are mostly interested in the following examples of tangles.

\begin{eg}\label{end_is_tangle}
 Each vertex-end $\omega$ induces a tangle; indeed, for a finite order separation $(A,B)$ we define 
$A$ to be big in this tangle if $\omega$ lives in $A$. It is straightforward to check that this 
defines a tangle of infinite order.
\end{eg}

Given two separations $(A,B)$ and $(C,D)$, the separation $(A\cap C, B\cup D)$ is called the 
\emph{corner separation} at the \emph{corner} $A\cap C$, see \autoref{fig:corner}. 
\begin{figure}
\begin{center}
   	  \includegraphics[height=3cm]{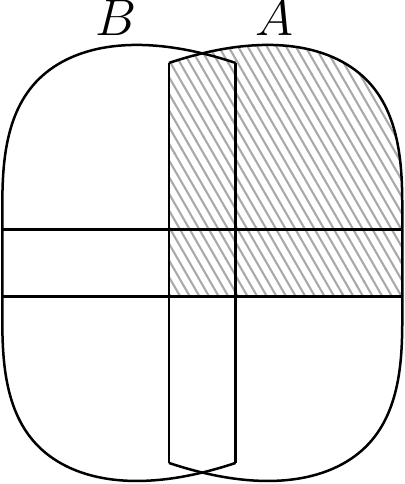}
   	  \caption{The corner diagram for the two separations $(A,B)$ and 
$(C,D)$. The separation $(A,B)$ separates vertically, while $(C,D)$ separates 
horizontally. The corner $A\cap C$ is shaded in grey. The middle region $A\cap B\cap C\cap D$ is 
called the \emph{center}. The four other regions `linking the corners' are called 
the \emph{links}. Formally, they are $(A\cap B)\sm C$,  $(A\cap B)\sm D$, $(C\cap D)\sm 
A$ and $(C\cap D)\sm B$. }\label{fig:corner}
\end{center}
   \end{figure}
In \autoref{fig:corner} the separator of $(A\cap C, B\cup D)$ has the shape of an `L'. Hence we 
denote this separator by $L(A,C)$; formally, $L(A,C)$ is the intersection of $A\cap C$ and $B\cup 
D$. The pair consisting of $(A,B)$ and $(C,D)$ has three more corner separations, corresponding to 
the corners of \autoref{fig:corner}. These are $(A\cap D, B\cup C)$, $(B\cap C, A\cup D)$ and 
$(B\cap D, A\cup C)$. Analogously to $L(A,C)$ we define the separators $L(A,D)$, $L(B,C)$ and 
$L(B,D)$. 

\begin{obs}\label{XXX}
 $|L(A,C)|+|L(B,D)|= |A\cap B|+|C\cap D|$.
\qed
\end{obs}

Given two separations $(A,B)$ and $(C,D)$ of order at most $k$ such that $L(A,C)$ contains at 
most $k$ vertices, then the corner separation $(A\cap C, B\cup D)$ has order at most $k$.
If additionally $P$ is a tangle of order $k+1$ such 
that $A$ and $C$ are big in $P$, then the side $A\cap C$ of $(A\cap C, B\cup D)$ is big in $P$; 
this follows from the first property of tangles as $B$, $D$ and $A\cap C$ cover all edges. 
We shall refer to that property of tangles as the \emph{corner property}. 

Another property of tangles of order $k+1$ that also follows from the property that no three small 
sides cover is that they are \emph{robust}\footnote{In the context of tangles and separations the 
term `robust' is used by different authors to mean different things that do not seem to be closely 
related. The notion we use was first defined in \cite{CDHH:profiles}. }; that is: given two 
separations $(A,B)$ and $(C,D)$, where the first separation has order at most $k$ and the second 
separation has arbitrary finite order such that the corner separators $L(A,C)$ and $L(B,C)$ have at 
most $k-1$ vertices. Then if the side $C$ is big, then one of the corners  $C\cap A$ or $C\cap B$ 
must be big.  

A separation $(A,B)$ \emph{distinguishes} two tangles $P_1$ and $P_2$ if the side 
$A$ is big in some $P_i$ and small in $P_{i+1}$. 
Note that $(A,B)$ distinguishes the $P_i$ if and only if $(B,A)$ distinguishes them. 
A separation distinguishes $P_1$ and $P_2$ 
\emph{efficiently} if 
it distinguishes them and has minimal order amongst all separations distinguishing them.

\subsection{Blocks and torsos}\label{blocks_and_seps}

Given a set $\Ncal$ of separations, an \emph{$\Ncal$-block} is a maximal set of vertices no two of 
which are separated\footnote{Two vertices $v_1$ and $v_2$ are \emph{separated} by a separation 
$(A,B)$ if some $v_i$ is in $A\sm B$ and $v_{i+1}$ is in $B\sm A$. } by a separation in $\Ncal$. 
For any $\Ncal$-block $\beta$, any separation in $\Ncal$ has (at least) one side that includes 
$\beta$. And $\beta$ can be written as an intersection of all these sides.

Let $\Ncal$ be a nested\footnote{A \emph{nested set} is a set of separations that are pairwise 
nested. A separation is \emph{nested} with a set if it is nested with every separation in that 
set. } set of separations of order at most $k$ and let $\beta$ be an 
$\Ncal$-block of at least $k+1$ vertices.  Let $P$ be a 
tangle of order $\ell+1$ greater than $k$. We say that the tangle $P$ \emph{lives} in the block 
$\beta$ 
if for every separation $(A,B)$ of $G$ of order at most $k$ with $\beta\se A$ the side $A$ is 
big in $P$. 

\begin{rem}
 Unlike for finite graphs, not every tangle of order $k+1$ of an infinite graphs lives in an 
$\Ncal$-block; 
indeed for tangles defined from ends the intersections of all big sides of separations in $\Ncal$ 
may be empty.  
\autoref{weird} shows that the definition of `lives in' cannot be weakened by replacing `$(A,B)$ of 
$G$' by `$(A,B)$ of $\Ncal$'. 
\end{rem}

\begin{eg}\label{weird}
In this example we construct a nested set of separations of order three such that the 
intersections of the big sides of the tangle forms a block of size four in which the tangle 
does not live.
 We obtain the graph $G$ from a ray by attaching vertices $v$ and $w$ complete to the ray and 
then attaching an edge complete to $v$ and $w$, see \autoref{fig:easy1}. 
\begin{figure}
\begin{center}
   	  \includegraphics[height=3cm]{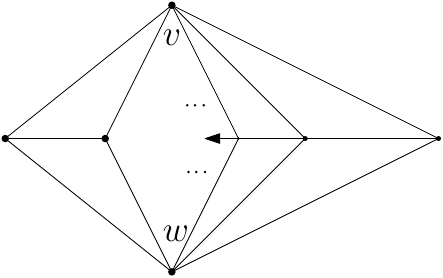}
   	  \caption{The graph $G$.}\label{fig:easy1}
\end{center}
   \end{figure}
The tangle we focus on is the tangle of the vertex-end of $G$. The set $\Ncal$ consists of 
those separations of the form $(P_n+v+w, G\sm P_n)$, where $P_n$ is 
the initial subpath of the ray of length $n$. The attached edge together with $v$ and $w$ is an 
$\Ncal$-block. This 
$\Ncal$-block is the intersection the big sides of separations in $\Ncal$. Still the tangle does not 
live in that block in the sense that it induces a tangle in that block in the sense of 
\autoref{induces_in_torso} below. 
\end{eg}

The next lemma gives a criterion when tangles do live in blocks.

\begin{lem}\label{extendable_lemma_rewritten}
 Let $\Ncal$ be a nested set of separations of order at most $k$ and $(A,B)$ a separation of order 
$\ell\geq k+1$ nested with $\Ncal$. Assume that $(A,B)$ distinguishes two tangles $P$ and $Q$ 
efficiently.
Then there is an $\Ncal$-block $\beta$ such that $A\cap B\se \beta$ 
and $P$ and $Q$ live in $\beta$. 
\end{lem}

\begin{proof}
Since the separation $(A,B)$ is nested with any separation in $\Ncal$, no such separation separates 
the 
vertex set $A\cap B$.
Note that $A\cap B$ contains at least $k+1$ vertices. 
Let $\beta$ be the unique $\Ncal$-block including $A\cap B$: as above $\beta$ is unique as each 
separation in $\Ncal$ has precisely one side containing $A\cap B$.

Next we show that $P$ and $Q$ live in $\beta$. For that let $(C,D)$ be a separation of order at 
most $k$ of $G$ with $\beta \se C$. Our aim is to show that $C$ is big in $P$ and $Q$. 

Suppose for a contradiction that the side $D$ is big in one of the tangles, say $P$. 
By symmetry, we may assume that the side $A$ is big in $P$. Since the link $(A\cap B)\sm C$ is 
empty, the corner separation $(A\cap D, B\cup C)$ has order at most $k$. As $P$ has the corner 
property, the corner $A\cap D$ is big in $P$. On the other hand, the side $B\cup C$ must be big in 
$Q$ as it includes the big side $B$. Hence the corner separation $(A\cap D, B\cup C)$ distinguishes 
$P$ and $Q$.
As this separation has order at most $k$, this is a contradiction to the efficiency of the 
separation $(A,B)$. Thus the side $C$ is big in both tangles $P$ and $Q$. 
\end{proof}

\begin{obs}\label{extendable_lemma_rewritten_2}
 In the proof of \autoref{extendable_lemma_rewritten} we do not make use of the whole 
strengths of the property of tangles that three small sides do not cover but just of 
the slightly weaker corner property. This will be used only once, namely in 
\autoref{extendable_lemma_rewritten_3}
\end{obs}

Given a set $\Ncal$ of separations and an $\Ncal$-block $\beta$, the \emph{torso} $G_T[\beta]$ of 
$\beta$ is obtained from $G[\beta]$ by adding 
an edge between any two vertices of $\beta$ that are in a common separator $A\cap B$ of some 
separation $(A,B)$ in $\Ncal$. 
This definition is compatible with the usual definition 
of torso \cite{DiestelBook10} in the context of tree-decompositions: if $\Ncal$ is the set of 
separations corresponding to the edges of a tree-decomposition, then the vertex set of every 
maximal 
part is an $\Ncal$-block and its torso is just the torso of that part. 
Moreover, we have the following.
\begin{lem}\label{torso_nice}
 Let $K$ be a component of $G-\beta$. Then any two vertices  $v$ and $w$ in the neighbourhood of 
$K$ in $\beta$ are adjacent in the torso $G_T[\beta]$. 
\end{lem}

\begin{proof}
 Let $P$ be a path between $v$ and $w$ whose interior vertices are in $K$. For each separation 
$(C,D)$ in $\Ncal$ its restriction to $P\cup \beta$ is $(C\cap (P\cup \beta), 
D\cap (P\cup \beta))$. 
By reversing separations in 
$\Ncal$ if necessary, we may assume that 
$\beta\se C'$ for every restriction $(C',D')$. 
Since nestedness is preserved by restricting, there is one such restriction $(C',D')$ such that 
$D'\sm C'$ includes all sets $D''\sm C''$ for all other such restrictions $(C'',D'')$. As no vertex 
of $P-v-w$ is in $\beta$ the set $D'\sm C'$ must be equal 
to $P-v-w$.
Hence $vw$ is an edge in the torso. 
\end{proof}

\begin{lem}\label{block_components_small}
 Let $\Ncal$ be a nested set of separations of order at most $k$ and let $\beta$ be an 
$\Ncal$-block. Then every component $C$ of $G-\beta$ has at most $k$ neighbours in $\beta$.
\end{lem}

\begin{proof}
This lemma is a well-known fact for finite graphs\footnote{Indeed, let $(A,B)$ be a separation in 
$\Ncal$ with a vertex $v$ of $C$ contained in the side of $(A,B)$ that does not include $\beta$ 
such that the side containing $v$ is inclusion-wise maximal. It is routine to check that the 
separator $A\cap B$ includes the neighbourhood of the component $C$. }.
We give an argument that reduces the infinite version to the finite version. 

Suppose for a contradiction that some component $C$ of $G-\beta$ has at least $k+1$ vertices in its 
neighbourhood. Then there is a finite connected subset $C'$ of $C$ that has a set $\beta'$ of $k+1$ 
vertices of $\beta$ included in its neighbourhood. We obtain the graph $G'$ from $G$ by deleting 
all vertices not in the finite vertex set $C'\cup \beta'$. The restrictions $(A',B')=(A\cap G', 
B\cap G')$ of separations $(A,B)$ in $\Ncal$ form a nested set of separations in $G'$. Hence we get 
the desired contradiction by the finite version of the lemma. This completes the proof.  
\end{proof}

\begin{rem}\label{why_profiles}
Tangles have many nice properties. However, they do not always induce tangles in blocks they live 
in, see \autoref{not_block} below. This property will be essential for our proof strategy later on. 
We will overcome that problem by working within the class of `robust profiles', a slight superclass 
of tangles.  

It should be noted that robust profiles unlike tangles do not always have the following property, 
which makes tangles work very well with graph minors: let $G'$ be a minor of $G$ and $T'$ be a 
tangle in $G'$, then there is a tangle in $G$ inducing $T'$.\footnote{Conversely, it can be shown 
that any profile in a graph $G'$ with the property that it induces a profile in any graph $G$ that 
has $G'$ as a minor is a tangle. } This last statement is not used 
in this paper. 
\end{rem}

\begin{eg}\label{not_block}
Consider the unique tangle of order $k+1$ on the complete bipartite graph 
$K_{k,k+1}$ for $k>3$. The separations of order $k$ are nested and the torso of the block in 
which the tangle lives is isomorphic to $K_k$. However, there is no tangle of order $k+1$ at $K_k$. 
(The largest tangle has order roughly $\frac{2}{3}\cdot k$.)  
\end{eg}

Robust profiles\footnote{Profiles were introduced in \cite{CDHH:profiles}.  In that 
paper `robust' is called `$\infty$-robust'. The results and proofs of this paper extend verbatim to 
`$r$-robust profiles' for any natural number $r$. The reader interested in such generalisation is 
refered to \cite{C:undom_td_v4}, an earlier version of this paper.}  will be defined like 
tangles except that we weaken the property that three small sides never cover; namely we just forbid 
this for very particular configurations. 
To be precise, we define \emph{robust profiles} like `tangles' except that we replace the 
first property that three small sides never cover all edges by the following three properties.
\begin{enumerate}
 \item no two small sides cover all edges;
 \item the corner property;
  \item the robustness property.
\end{enumerate}

\begin{eg}
 We have seen above that tangles are examples of robust profiles. A different example is the robust 
profile of order $k+1$ on the graph $K_k$.
\end{eg}

All definitions for tangles are extended to robust profiles in the obvious way. 
The proof of \autoref{induces_in_torso} is the only one in the paper where we make use of the 
difference between tangles and robust profiles (except from those implicit places where we apply 
\autoref{induces_in_torso}). This is necessary in order to cope with examples such as those in 
\autoref{not_block}. 

Next we define how a robust profile living in an $\Ncal$-block $\beta$ defines a robust profile in 
the torso 
graph $G_T[\beta]$. The \emph{restriction} of a separation $(A,B)$ of $G$ 
to $\beta$ is the 
separation $(A\cap \beta, B\cap \beta)$ of $G[\beta]$. 

\begin{lem}\label{restriction_well_def}
 Given an $\Ncal$-block $\beta$ and a separation $(A,B)$ nested with $\Ncal$, the restriction of 
$(A,B)$ is a separation in the torso graph $G_T[\beta]$.
\end{lem}

\begin{proof}
It suffices to show that for any separation $(C,D)\in \Ncal$ that $C\cap D$ is a subset of $A$ or 
$B$. This follows from the nestedness of $(A,B)$ with $(C,D)$.  
\end{proof}

For any separation $(A',B')$ of a torso graph $G_T[\beta]$, there is a separation $(A,B)$ of $G$ 
that restricts to $(A',B')$ and has the same separator. 
Now let $P$ be a robust profile of order $\ell+1 > k$ that lives in an $\Ncal$-block $\beta$.
The \emph{induced robust profile} $P_\beta$ of $P$ at $\beta$ is defined as follows.
A side $A'$ of a separation $(A',B')$ of the torso graph $G_T[\beta]$ of order at most $\ell$ is 
\emph{big in $P_\beta$} if and only if there is a side $A$ of a separation $(A,B)$ of 
$G$ that restricts to $(A',B')$ and has the same separator such that $A$ is big in $P$.  

\begin{lem}\label{induces_in_torso}
 Assume that a robust profile $P$ of order $\ell+1>k$ lives in the $\Ncal$-block $\beta$. Then the 
induced 
robust profile $P_\beta$ is a robust profile of the torso $G_T[\beta]$. 
\end{lem}

\begin{proof}

First we show that if $(A',B')$ is a separation of order at most $k$ of the torso, then it can 
have at most one big side in $P_\beta$. 

By the component property, there is a component $K$ of the graph $G-A'\cap B'$ such that 
the side $K\cup(A'\cap B')$ is big in $P$. As $P$ lives in $\beta$ by assumption, 
the block $\beta$ is included in that side. As $\beta$ has at least $k+1$ vertices,  
the component $K$ contains a vertex of the block $\beta$. 
That is, the vertex set $K'=K\cap \beta$ is not empty. 

As $K'$ is a restriction of a connected set, it is connected in the torso by \autoref{torso_nice}. 
As the vertex set $K'$ 
is disjoint from the separator $A'\cap B'$, there is a unique side of the separation $(A',B')$ that 
includes $K'$, say $A'$. Now let $(A,B)$ be any separation of $G$ that restricts to $(A',B')$ and 
has the separator $A'\cap B'$. Since $K$ includes $K'$, the set $K$ cannot be included in $B$. So 
it is included in $A$. So $A$ must be big in $P$ as it includes a big side. In particular 
$B$ is small in $P$. Since $(A,B)$ is 
arbitrary, $B'$ cannot be big.

To see that $P_\beta$ has the component property, let $X$ be a set of at most $\ell$ 
vertices of the torso.
Let $K$ be a component of $G-X$ such that the side $(K\cup X)$ of $(K\cup X, G\sm K)$ is big in 
$P$. Let $K'=K\cap \beta$, which is connected in the torso by \autoref{torso_nice}. 

Next we show that $K'$ is not empty. 
Suppose not for a contradiction. Then the component $K$ contains no vertex of the block 
$\beta$. So $K$ is a component of $G-\beta$. By \autoref{block_components_small}, the component $K$ 
has at most $k$ neighbours in the block $\beta$. So the separation $(K\cup N(K), G\sm K)$ has order 
at most $k$. As $P$ lives in $\beta$, the side $G\sm K$ is big in $P$. This is a contradiction to 
the fact that $P$ is a robust profile as the side $K\cup N(K)$ is big in $P$. Hence the set $K'$ 
must be 
nonempty. 

Let $K''$ be the component of the torso $G_T[\beta]$ without $X$ including the connected 
nonempty set $K'$. Then $K''\cup X$ is big in the induced robust profile. So $P_\beta$  
has the component property.

It remains to show that two small sides in the torso do not cover and to show the corner property 
and robustness for $P_\beta$. 
To see the first, suppose for a contradiction that the torso is covered by two small sides. We 
observe that if the edge set of a complete graph is covered by two subgraphs, then one of 
these subgraphs must include the whole vertex set of the complete graph.\footnote{This 
observation is no longer true if we consider covers by three subgraphs instead; and is the 
reason why this proof does not work for tangles (which is not surprising in view of 
\autoref{not_block}).} Hence by \autoref{torso_nice} for each component $K$ of $G$ without the 
torso, there is one of the sides that includes the whole neighbourhood of $K$.
So we can assign each component $K$ to a side that includes its neighbourhood. Each of the two 
covering sides together with its assigned components forms a side of a separation of order at most 
$\ell$ in the graph $G$. As this side restricts to a small side in the torso, it must be small in 
the original robust profile $P$ by definition of the induced robust profile and by the first part 
of the proof. Hence the graph $G$ is covered by two sides that are small in $P$, which is not 
possible as $P$ is a robust profile $P$.  

Having shown that two small sides cannot cover in the torso, it remains to verify the the corner 
property and robustness for $P_\beta$. In a nutshell, they are both true as taking the corner 
separation commutes with taking the torso. In detail, let $(A',B')$ and $(C',D')$ be two 
separations of the torso of order at most $\ell$ and assume that their corner separator $L(A',C')$ 
contains at most $\ell$ vertices. Then there are separations $(A,B)$ and $(C,D)$ of $G$ that 
restrict to $(A',B')$ and $(C',D')$ and have the same separator; in particular, all vertices 
of the separators $A\cap B$ and $C\cap D$ are vertices of the torso. Hence the corner separator 
$L(A,C)$ is equal to the corner separator  $L(A',C')$. So the corner property for $P_\beta$ follows 
from the corner property for $P$. Similarly robustness for $P_\beta$ follows from robustness for 
$P$. So $P_\beta$ is a robust profile of the torso. 
\end{proof}

\begin{obs}\label{extendable_lemma_rewritten_3}
 \autoref{extendable_lemma_rewritten} is true with `tangle' replaced by `robust profile'. 
\end{obs}

\begin{proof}
 This follows from \autoref{extendable_lemma_rewritten_2}.
\end{proof}

\subsection{Extending separations of the torsos}\label{torsos_to_G}

The aim of this subsection is to explain how for a given nested set $\Ncal$ of separations and 
a torso of an $\Ncal$-block, a nested set of separations of the torso can be extended to a nested 
set of separations of the whole graph that is nested with $\Ncal$. This is more technical 
and hence more complicated as one might expect. Indeed, extending a single separation of the torso 
is quite easy -- but it is not uniquely defined. We have to make some choices. If we make these 
choices arbitrarily for two nested separations, it could happen that their extensions are no longer 
nested, see \autoref{no_longer_nested} below. 

Throughout this subsection we fix a nested set $\Ncal$ of separations and an $\Ncal$-block $\beta$. 
For each separation $(C,D)\in \Ncal$ at least one of the sides $C$ and $D$ includes $\beta$. 
Let $\Ncal_\beta$ consist of those separations $(C,D)$ such that $\beta$ is included in $C$ and 
$(C,D)$ or $(D,C)$ is in $\Ncal$. 

Given a separation $(A,B)$ of the torso $G_T[\beta]$, one way to `extend' $(A,B)$ to a separation 
of $G$ is to decide for each component of $G-\beta$, whether we put it on the $A$-side or on the 
$B$-side. Below we define what it means when such a component is `forced'. Informally, it is forced 
when we must put it on the $A$-side in order to extend $(A,B)$ to a separation of $G$. 

A component $K$ of $G-\beta$ is \emph{forced at step $1$} by $(A,B)$ if one of 
its vertices has a neighbour in $A\sm B$.
A separation $(C,D)\in \Ncal_\beta$ is \emph{forced at step $2n+2$} if there is a component $K$ 
forced at step $2n+1$ that contains a vertex of $D\sm C$.
A component $K$ of $G-\beta$ is \emph{forced at step $2n+1$} for $n>0$ if there is a separation 
$(C,D)\in \Ncal_\beta$ forced at step $2n$ so that $K$ contains a vertex of $D\sm C$. 
An alternative definition of `forcing' is the following. 
\begin{eg}
 We define the bipartite graph whose left side are the components of $G-\beta$ and whose right side 
are the separations in $\Ncal_\beta$. We add an edge between a component $K$ and a separation 
$(C,D)$ if $K$ contains a vertex of $D\sm C$. A component (or separation) is forced if and only if 
its connected component in this bipartite graph contains a component forced at step one. We will 
not use the fact that this definition is equivalent in our proofs. 
\end{eg}

The following lemma implies that if a component is forced at some step, it is forced at step one or 
three; and if a separation is forced, it is forced at step two or four.

\begin{lem}\label{forces}
 Let $(C,D)$ be a separation in $\Ncal_\beta$ forced by $(A,B)$. 
 There is some $(C',D')$ in $\Ncal_\beta$ forced by $(A,B)$ with $D\sm C\se D'\sm C'$ such 
that 
some vertex of $C'\cap D'$ is in $A\sm B$. 
\end{lem}

\begin{proof}
Let $2n$ be the smallest step at which $(C,D)$ is forced. 
We prove \autoref{forces} by induction on $n$. 

The base case is that $2n=2$. Let $K$ be a component forced at step one `forcing' $(C,D)$; 
here we say that $K$ \emph{forces} the separation $(C,D)$ if there is a vertex of $K$ in $D\sm C$ 
and $(C,D)$ is not forced at an earlier step than $K$. 

As $K$ is forced at step one, there is a vertex $v$ of $K$ that has a neighbour $w$ in $A\sm B$.
As $v$ is not in $\beta$, there is a separation $(E, F)$ in $\Ncal_\beta$ such that $v$ is in 
$F\sm E$. As $w$ is in $\beta$, it is in $E$. As it has a neighbour in $F\sm E$, it also must 
be in $F$. 

We call a separation $(E',F')$ a \emph{candidate} if the separator $E'\cap F'$ contains a vertex of 
$A\sm B$ and $F'\sm E'$ contains a vertex of the component $K$. 
For example, the separation $(E,F)$ is a candidate. 
To conclude the base case, we show 
the following. 

\begin{sublem}\label{sublem_new}\label{nested_calc}
 Assume that there is a candidate. Then there is a separation $(C',D')$ in $\Ncal_\beta$ 
forced by $(A,B)$ with $D\sm C\se D'\sm C'$ such that 
some vertex of $C'\cap D'$ is in $A\sm B$. 
\end{sublem}

\begin{proof}
We pick a candidate $(E,F)$. Let $w$ be a vertex of the separator $E\cap F$ in $A\sm B$. 
If the vertex $w$ was in the separator $C\cap D$, the lemma would be true with `$(C,D)$' in place 
of `$(C',D')$'.
Hence we may assume that the vertex $w$ of $\beta$ is not in the side $D$ as $\beta\se C$. 
So the vertex $w$ is in the link $(E\cap F) \sm D$. 
So from the nestedness of $(C,D)$ and 
$(E,F)$ it follows that either  $D\sm C\se F\sm E$ or else  $D\sm C$ and  $F\sm E$ are 
vertex-disjoint. 

Our aim is to construct a candidate $(E,F)$ that satisfies the first condition $D\sm C\se F\sm E$.
Let $u$ and $v$ be vertices of $K$ that are in $D\sm C$ and $F\sm E$, respectively (such vertices 
exist as we may assume that $D\sm C$ is not empty and $(E,F)$ is a candidate).. 
Let $P$ be a path from the vertex $u$ to 
vertex $v$ included in the component $K$ of $G-\beta$. By assumption for every vertex $x$ on $P$, 
there is a separation $(C_x,D_x)\in \Ncal_\beta$ with $x\in D_x\sm C_x$. 
If possible we choose the separation $(C_x,D_x)$ such that the vertex $w$ is in the separator (in 
that case it is a candidate). 

Our goal is to show that it is possible to choose the separation  $(C_u,D_u)$ at $u$ such 
that the vertex $w$ is in the separator. 
Indeed, then we can use the nestedness of $(C_u,D_u)$ and $(C,D)$ to deduce as above that $D\sm 
C\se D_u\sm C_u$ or else  $D\sm C$ and  $D_u\sm C_u$. However, here the second outcome is not 
possible as the vertex $u$ is in the intersection of these two sets. 

Suppose for a contradiction that such a choice for $(C_u,D_u)$ is not possible. Let $x$ be the 
vertex  on the path $P$ nearest to $u$ 
such that the vertex $w$ is in its separator $C_x\cap D_x$. Let $y$ be the neighbour of $x$ on $P$ 
nearer to $u$, which exists as $x\neq u$. Then the vertex $w$ is in the link $(C_x\cap D_x)\sm 
D_y$. As above we deduce from 
the nestedness of the separations $(C_x,D_x)$ and $(C_y,D_y)$, that either $D_y\sm C_y\se D_x\sm 
C_x$ or else  $D_y\sm C_y$ and  $D_x\sm C_x$ are 
vertex-disjoint. Since we cannot choose $(C_x,D_x)$ in place of $(C_y,D_y)$, the first outcome 
is impossible. The second outcome is not possible either as the 
vertex $x\in D_x\sm C_x$ is adjacent to the vertex $y\in D_y\sm C_y$. So this is the desired 
contradiction. Hence we can choose $(C_u,D_u)$ such that it is a candidate, which completes the 
proof as shown above. 
\end{proof}

So the base case follows from \autoref{sublem_new} and the fact that $(E,F)$ is a candidate. 

Now let $n>1$ and assume that we already proved \autoref{forces} for separations $(E,F)$ forced 
at some step before $2n$. 
Let $(C,D)$ be a separation in $\Ncal_\beta$ forced at step $2n$. Let $K$ force $(C,D)$. 
Let $(C_1,D_1)$ be a separation forcing $K$, which 
exists as $n>1$. By the induction hypothesis, there is a separation $(C_2,D_2)$ in $\Ncal_\beta$ 
forced by $(A,B)$ with $D_1\sm C_1\se D_2\sm C_2$ such that some 
vertex $w$ of $C_2\cap D_2$ is in $A\sm B$. As $D_1\sm C_1\se D_2\sm D_2$, there is a vertex $v$ of 
$K$ that is in $D_2\sm C_2$. So $(C_2,D_2)$ is a candidate. So the induction step follows from 
\autoref{sublem_new}. This completes the proof. 
\end{proof}

\begin{lem}\label{C-forced_compatible}
 For any separation $(A,B)$ of the torso, no component 
$K$ of $G-\beta$ is forced by  both $(A,B)$ and  $(B,A)$.
\end{lem}

\begin{proof}
As any component of $G-\beta$ forces some separation in $\Ncal_\beta$, it suffices to show 
that no separation  $(C,D)$ in $\Ncal_\beta$ is forced by both $(A,B)$ and $(B,A)$. 
Suppose for a contradiction that there is such a separation $(C,D)$. 

By \autoref{forces}, there is a separation $(C',D')\in \Ncal_\beta$ forced by $(A,B)$ with 
$D\sm C\se D'\sm C'$ such that some 
vertex $v$ of $C'\cap D'$ is in $A\sm B$. As $D'\sm C'$ is a superset of $D\sm C$, the separation 
$(C',D')$ is also forced by $(B,A)$. Applying \autoref{forces} to $(C',D')$ and to $(B,A)$, yields 
a separation $(C'',D'')$ with $D'\sm C'\se D''\sm C''$ such that some 
vertex $w$ of $C''\cap D''$ is in $B\sm A$. Since the vertex $v$ is in $\beta$, it must be in 
$C''$. As $D''$ includes $D'$, it also is in $D''$. In short, $v$ is in the separator $C''\cap 
D''$. 

Hence the separation $(C'',D'')$ witnesses that $vw$ is an edge of the torso. As $v$ is in $A\sm 
B$ and $w$ is in $B\sm A$, we deduce that $(A,B)$ cannot be a separation of the torso. That is 
the desired contradiction.  
\end{proof}

\vspace{.3cm}

Having finished the proof of \autoref{C-forced_compatible}, we now define naive extensions of 
separations of the torso, explain why they are not quite the object we need and define extensions 
of nested sets of separations of the torso. 

Given a separation $(A,B)$ of the torso, the side $\hat A$ is obtained from $A$ by adding all 
components $K$ of $G-\beta$ that are forced at some step. 
We obtain $\widehat{\hat B}$ from $B$ by adding all 
components $K$ that are not forced at any step. Note that $\widehat{\hat B}= B \cup 
(G\sm \hat 
A)$. 
We define the \emph{naive extension} of $(A,B)$, denoted by $\widehat{(A,B)}$, to be $(\hat A, 
\widehat{\hat B})$. This construction ensures that  $\widehat{(A,B)}$ is a separation of $G$ that 
restricts to $(A,B)$. 

\begin{rem}
We chose the notation `$\widehat{\hat B}$' instead of simply `$\hat B$' as the term 
$\hat A$ for the separation $(A,B)$ and the term `$\widehat{\hat A}$' for the separation $(B,A)$ 
need not a priori agree -- and in fact they do not agree for $(A,B)$ defined as in 
\autoref{no_longer_nested}. 

In particular, the reverse separation of $\widehat{(A,B)}$ is in general not equal to 
$\widehat{(B,A)}$. 
\end{rem}

\begin{obs}\label{monotone}
Given two separations $(A,B)$ and $(X,Y)$ of the torso, if $(A,B)\leq (X,Y)$, then 
$ \widehat{(A,B)}\leq \widehat{(X,Y)}$. 
\qed
\end{obs}

\begin{obs}\label{Y_hat_nested}
Let $(A,B)$ be a separation of the torso and 
$(C,D)\in \Ncal_\beta$ be a proper separation. Then $ \widehat{(A,B)}$ or its reverse separation 
$(\widehat{\hat 
B},\hat A)$ is $\leq (C,D)$.

In particular, $ \widehat{(A,B)}$ is nested with every proper separation of $\Ncal$. 
\end{obs}

\begin{proof}
First assume that the separation $(C,D)$ is forced at some step. Then $D\sm C$ is a subset of 
$\hat A$. Since the separator $A\cap B$ of the separation $(\widehat{\hat 
B},\hat A)$ is a subset of the block $\beta$, which is included in the side $C$, we conclude that 
$D\sm C$ is a subset of 
$\hat A\sm \widehat{\hat B}$. As in the proof of \autoref{proper_nested} one combines this with the 
assumption that the separation $(C,D)$ is proper to deduce that $(D,C)\leq  \widehat{(A,B)}$. 

Hence it remains to consider the case that the separation $(C,D)$ is not forced at any step. 
Analoguously as above, one shows that $(D,C)\leq (\widehat{\hat 
B},\hat A)$ in that case. So $ \widehat{(A,B)}$ or its reverse separation 
$(\widehat{\hat 
B},\hat A)$ is $\leq (C,D)$.

\end{proof}

\autoref{no_longer_nested} gives an example of nested separations $(A,B)$ and $(C,D)$ of the torso 
$G_T[\beta]$ whose naive extensions $\widehat{(A,B)}$ and $ \widehat{(C,D)}$ are not nested. 

\begin{eg}\label{no_longer_nested}
Let $G$ be the labelled graph depicted in \autoref{fig:naive}. The set $\Ncal$ consists of the 
separation of order one and its reverse. Then the torso is $G-6$. We define $A=\{1,2,3,5\}$, 
$B=\{3,4,5\}$, $C=\{2,3,4,5\}$, $D=\{1,2,5\}$.
Then $(A,B)$ and $(C,D)$ are nested but not $\widehat{(A,B)}$ and $ \widehat{(C,D)}$.
 \begin{figure}
\begin{center}
   	  \includegraphics[height=3cm]{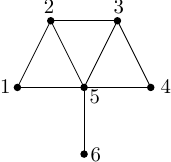}
   	  \caption{The graph $G$ is obtained by from a triangle by attach two 
more triangles at distinct edges and by then attaching a leaf at the unique 
vertex of degree four.}\label{fig:naive}
\end{center}
\end{figure}
\end{eg}
Examples like \autoref{no_longer_nested} motivate the slightly technical definition of 
$\widetilde{ \Lcal}$ below.
Given a nested set $\Lcal$ of separations of $G_T[\beta]$, the \emph{extension} $\widetilde{ 
\Lcal}$ of $\Lcal$ (depending on a well-order $( {(A_\alpha, B_\alpha)}\mid \alpha\in \kappa)$ of 
$\Lcal$) is the set $\{\widetilde {(A_\alpha, B_\alpha)} \mid (A_\alpha, B_\alpha)\in \Lcal\}$, 
where the 
\emph{extension} $\widetilde {(A,B)} $ of $(A,B)$ is defined as follows:
for the smallest element $(A_0,B_0)$ of the well-order, we just let 
$\widetilde {(A_0,B_0)}= \widehat{(A_0,B_0)}$.

Assume that we already defined $\widetilde {(A_\alpha, B_\alpha)}$ for all 
$\alpha<\gamma$. 
A component $K$ of $G-\beta$ is \emph{$\gamma$-forced} if there is some $\alpha<\gamma$
such that $K$ is a subset of $\tilde B_\alpha$ and $(B_\alpha, A_\alpha)\leq (A_\gamma, B_\gamma)$.
We obtain $\tilde A_\gamma$ from $A_\gamma$ by adding all components $K$ of $G-\beta$ that are 
forced by $(A_\gamma, B_\gamma)$ or are $\gamma$-forced.
We obtain $\tilde B_\gamma$ from $B_\gamma$ by adding all other components.
The extension $\widetilde{(A_\gamma,B_\gamma)}$ of $(A_\gamma,B_\gamma)$ is defined to be 
$(\tilde A_\gamma,\tilde B_\gamma)$. 

\begin{eg}
 The nested set $\Lcal=\{(A,B), (C,D)\}$ defined in \autoref{no_longer_nested} has different 
extensions $\widetilde \Lcal$ depending on which well-order we choose.
\end{eg}

\begin{obs}\label{welldef1}
 The extension $(\tilde A_\gamma,\tilde B_\gamma)$ is a separation.
\end{obs}

\begin{proof}
It suffices to show that any component $K$ of $G-\beta$ included in $\tilde A_\gamma$ is not forced 
by $(B_\gamma,A_\gamma)$. 
By \autoref{C-forced_compatible}, we may assume that $K$ is $\gamma$-forced. 
As any class of ordinals has a least element, there is some $\alpha$ minimal such that $K$ is a 
subset of $\tilde B_\alpha$ and $(B_\alpha, A_\alpha)\leq (A_\gamma, B_\gamma)$. In particular, $K$ 
is not forced by $(A_\alpha, B_\alpha)$. As $B_\gamma\sm A_\gamma$ is a subset of $A_\alpha\sm 
B_\alpha$, we deduce that 
$K$ cannot be forced by  $(B_\gamma,A_\gamma)$. 
\end{proof}

\begin{obs}\label{same_order}
Any separation $(A,B)$ of the torso has the same separator as its extension $\widetilde{(A,B)}$.\qed
\end{obs}

\begin{obs}\label{welldef2}
For any two separations $(A,B)$ and $(B,A)$ in $\Lcal$, the extension $\widetilde{(A,B)}$ is 
the reverse of 
$\widetilde{(B,A)}$.
\end{obs}

\begin{proof}
 We may assume that $(A,B)=(A_\alpha,B_\alpha)$ and $(B,A)=(A_\gamma,B_\gamma)$ for some 
$\alpha<\gamma$. It suffices to show that $\tilde B_\alpha= \tilde A_\gamma$. 
By construction $\tilde B_\alpha\se \tilde A_\gamma$. 
Suppose for a contradiction that $\tilde B_\alpha$ is a proper subset of $\tilde A_\gamma$.
Then there is a component $K$ that is included in $\tilde A_\alpha$ and in $\tilde A_\gamma$. 
We split into four cases and derive a contradiction in each of them.

{\bf Case 1A:} $K$ is forced by $(A,B)$ and $(B,A)$. This is impossible by 
\autoref{C-forced_compatible}.

{\bf Case 1B:} $K$ is forced by $(A,B)$ and $\gamma$-forced. 
So there is some ordinal $\delta<\gamma$ such 
that $K$ is a 
subset of $\tilde B_\delta$ and $(B_\delta, A_\delta)\leq (A_\gamma, B_\gamma)$.
Then $(A,B)\leq (A_\delta, B_\delta)$. So $K$ is forced by $(A_\delta, B_\delta)$. So it cannot be 
a subset of $\tilde B_\delta$, a contradiction.

{\bf Case 2A:} $K$ is $\alpha$-forced  and forced by $(B,A)$. This case is analogue to Case 1B. 

{\bf Case 2B:} $K$ is $\alpha$-forced  and $\gamma$-forced.
So there is some ordinal $\delta<\alpha$ such 
that $K$ is a 
subset of $\tilde B_\delta$ and $(B_\delta, A_\delta)\leq (A_\alpha, B_\alpha)$; and there is 
some ordinal $\epsilon<\gamma$ such 
that $K$ is a 
subset of $\tilde B_\epsilon$ and $(B_\epsilon, A_\epsilon)\leq (A_\gamma, B_\gamma)$.
To summarise:
\[
 (B_\delta, A_\delta)\leq (A,B)\leq (A_\epsilon, B_\epsilon)
\]

If $\delta< \epsilon$, then the component $K$ is $\epsilon$-forced and hence not in 
$\tilde B_\epsilon$, which is impossible. Similarly, we also cannot have $\delta> \epsilon$.
So $\delta= \epsilon$. But then the component $K$ is included in the sides $B_\delta$ and 
$A_\delta\supseteq B_\epsilon$, which is the desired contradiction. 
\end{proof}

\begin{obs}\label{welldef3}
For any two separations $(A_\alpha,B_\alpha)$ and $(A_\gamma,B_\gamma)$ in the nested set $\Lcal$,
their extensions $\widetilde{(A_\alpha,B_\alpha)}$ and $\widetilde{(A_\gamma,B_\gamma)}$ are 
nested.
\end{obs}

\begin{proof}
By symmetry we may assume that $\alpha<\gamma$.
If $B_\alpha$ is a subset of $A_\gamma$, it 
follows immediately from the definitions that $\widetilde{(B_\alpha,A_\alpha)}\leq 
\widetilde{(A_\gamma,B_\gamma)}$. If $A_\alpha$ is a subset of $A_\gamma$, then 
$\widetilde{(A_\alpha,B_\alpha)}\leq \widetilde{(A_\gamma,B_\gamma)}$ by construction.

The other two cases can be deduced using \autoref{welldef2} as follows.
First assume that $(B_\gamma,A_\gamma)$ is not in $\Lcal$. Then we add that separation at the 
end of the well-order for $\Lcal$. Now we apply the above argument to $(A_\alpha,B_\alpha)$ and 
$(B_\gamma,A_\gamma)$. Hence $\widetilde{(A_\alpha,B_\alpha)}$ and 
$\widetilde{(B_\gamma,A_\gamma)}$ are nested. By \autoref{welldef2} also 
$\widetilde{(A_\alpha,B_\alpha)}$ and $\widetilde{(A_\gamma,B_\gamma)}$ are 
nested.

The same argument works if $(B_\gamma,A_\gamma)$ is in $\Lcal$ but in the well-order after position 
$\alpha$. If it is before $\alpha$, we replace $(A_\alpha,B_\alpha)$ or $(A_\gamma,B_\gamma)$ by 
their reverses if they appear before in the well-order and then do the above argument. This implies 
the desired result by \autoref{welldef2}. 
\end{proof}

\begin{obs}\label{welldef4}
For any separation $(A,B)\in \Lcal$, its extension $\widetilde{(A,B)}$ is nested with every proper 
separation in $\Ncal$. 
\end{obs}

\begin{proof} Let $(A,B)=(A_\gamma,B_\gamma)$.
We say that a separation $(C,D)$ of $\Ncal_\beta$ is \emph{$\gamma$-forced} if there is some 
component $K$ of $G-\beta$ that contains a vertex of $D\sm C$ and is $\gamma$-forced or forced by 
$(A_\gamma,B_\gamma)$. 

We claim that if a separation of $\Ncal_\beta$ 
is $\gamma$-forced, then every component of $G-\beta$ that contains a vertex of $D\sm C$ is 
$\gamma$-forced or forced by 
$(A_\gamma,B_\gamma)$. Indeed, if any such component is forced by a separation of the nested 
set $\Lcal$, then all of these components are. Hence this follows by transfinite induction on the 
well-order of $\Lcal$.

Using this, we can argue as in the proof of \autoref{Y_hat_nested}.
\end{proof}

\begin{lem}\label{about_different_blocks}
Let $\Ncal$ be a nested set of proper separations and let $\beta$ and $\gamma$ be distinct 
$\Ncal$-block.
Let $\Lcal_\beta$ and $\Lcal_\gamma$ be nested sets of separations of $G_T[\beta]$ and 
$G_T[\gamma]$, respectively.
Then $\widetilde{ \Lcal}_\beta$ is a set of nested separations. 
For any separations $(A,B)\in \Lcal_\beta$ and  $(C,D)\in \Lcal_\gamma$, their  
extensions $\widetilde{(A,B)}$ and $\widetilde {(C,D)} $ are 
nested.
Moreover, they are nested with every separation in $\Ncal$. 
\end{lem}

\begin{proof}
The set $\widetilde{ \Lcal}_\beta$ is nested by \autoref{welldef3}. The 
`Moreover'-part follows from \autoref{welldef4}.
So it remains to show that for any separations $(A,B)\in \Lcal_\beta$ and  $(C,D)\in \Lcal_\gamma$, 
the 
extensions $\widetilde{(A,B)}$ and $\widetilde {(C,D)} $ are 
nested.

Since the blocks $\beta$ and $\gamma$ are distinct, there is a separation $(E,F)$ of $\Ncal$ such 
that one side includes the block $\beta$ and the other side includes the block $\gamma$.
By symmetry we may assume that $\beta$ is included in $E$ and $\gamma$ is included in $F$.
By \autoref{welldef4} $\widetilde{(A,B)}$ is nested with $(E,F)$. 
An argument as in the proof of \autoref{welldef4} gives that $F\sm E$ is included in $B\sm A$ or 
$A\sm B$. So either $\widetilde{(A,B)}$ or its reverse is $\leq (E,F)$. 

By \autoref{welldef2} it would be enough to show that one of $\widetilde{(A,B)}$ or its reverse is 
nested with $\widetilde {(C,D)} $. Hence by replacing `$(A,B)$' by $(B,A)$ if necessary, we assume 
that $\widetilde{(A,B)}\leq (E,F)$. Similarly, one may assume that 
$(E,F)\leq \widetilde{(C,D)}$. Combining this yields that $\widetilde{(A,B)}$ and 
$\widetilde {(C,D)} $ are nested.
\end{proof}

\begin{obs}\label{from_torso_to_G}
Let $\beta$, $P$, $Q$, $P_\beta$ and $Q_\beta$ as in \autoref{extendable_lemma_rewritten_3}.
Let $\Lcal$ be a nested set of separations in $G_T[\beta]$.
If a separation $(C,D)\in \Lcal$ distinguishes the induced robust profiles $P_\beta$ and $Q_\beta$ 
in the 
torso  $G_T[\beta]$, then the extension $\widetilde{(C,D)}$ 
distinguishes the robust profiles $P$ and $Q$. 
\end{obs}

\begin{proof}By symmetry we may assume that $C$ is big in $P_\beta$ and $D$ is big in $Q_\beta$.
As $P_\beta$ is a robust profile by \autoref{induces_in_torso}, the component property yields that 
there is 
a component $K_1$ of $\beta-(C\cap 
D)$ such that $K_1\cup (C\cap D)$ is big in $P_\beta$. 
So $K_1$ is a subset of $C$. 

The extension $\widetilde{(C,D)}=(\tilde{C}, \tilde{D})$ has the separator $C\cap D$ by 
\autoref{same_order}. Let $K_1'$ be the components of $G-C\cap D$ such that the side $K_1'\cup 
(C\cap D)$ is 
big in $P$.  As $P_\beta$ is induced by $P$, 
it must be that $K_1$ contains a vertex of $K_1'$. In particular $K_1'$ cannot be a subset of 
$\tilde D$. So it must be a subset of $\tilde C$. 
Hence $\tilde C$ is big in $P$. Similarly one 
shows that $\tilde D$ is big in $Q$. So the extension $\widetilde{(C,D)}$ 
distinguishes the robust profiles $P$ and $Q$. 
\end{proof}

\subsection{Miscellaneous}

The lemmas summarised in this subsection are well-known. 

\begin{lem}\label{are_nested}
Let $(A,B)$ and $(C,D)$ be proper separations such 
that $A\sm B$ is connected and does not intersect the separator $C\cap D$. Then $(A,B)$ and $(C,D)$ 
are nested.
\end{lem}

\begin{proof}
By the definition of nestedness, it suffices to show that $(A,B)\leq (C,D)$ or $(A,B)\leq (D,C)$.
As the connected set $A\sm B$ does not intersect the separator $C\cap D$, it is included in $C\sm 
D$ or $D\sm C$.
By symmetry, we may assume that is is included in $C\sm D$. So $A\sm B$ is included in $C\sm D$. 
Hence by \autoref{proper_nested} $(A,B)$ and 
$(C,D)$ are nested. 
\end{proof}

\begin{lem}[{\cite[Lemma 2.2]{CDHH:k-blocks}}]\footnote{This is Lemma 2.2 of that paper with the 
roles of `$(C,D)$' and 
`$(E,F)$' interchanged.}\label{corner_nested}
Let $(A,B)$, $(C,D)$ and $(E,F)$ be proper separations such that first $(A,B)$ and 
$(C,D)$ are not nested and second the corner separation $(A\cap C,B\cup D)$ is not nested 
with $(E,F)$. 
Then $(E,F)$ is not nested with $(A,B)$ or $(C,D)$.
\end{lem}

A separation $(A,B)$ of a graph $G$ is \emph{tight} if every component 
of $G$ without the separator $A\cap B$ has the whole separator $A\cap B$ in its neighbourhood.

\begin{lem}\label{many_comp_nested}
Let $(A,B)$ be a separation of order at most $k$.
Let $(C,D)$ be a tight separation such that the graph $G$ without the separator $C\cap D$ has at 
least $k+1$ components.
Then one of the links $(C\cap D)\sm A$ or $(C\cap D)\sm B$ is empty.
\end{lem}

\begin{proof}
Suppose not for a contradiction, then there are vertices $v\in (C\cap D)\sm A$ and $w\in (C\cap 
D)\sm B$. Then $v$ and $w$ are in the neighbourhood of every component of  $G$ without the 
separator $C\cap D$. Thus 
there are $k+1$ internally disjoint paths from $v$ to $w$.
All of these paths contain vertices of the separator $A\cap B$.
This contradicts the assumption that the separator $A\cap B$ contains at most $k$ vertices.
\end{proof}

Given two vertices $v$ and $w$, a separator $S$ separates $v$ and $w$ \emph{minimally} if each 
component of $G-S$ containing $v$ or $w$ has the whole of $S$ in its neighbourhood.

\begin{lem}[{\cite[Statement 2.4]{Halin:cuts}}]\label{dk}
Given vertices $v$ and $w$ and $k\in \Nbb$, there are only finitely many distinct separators of 
size 
at most $k$ separating $v$ from $w$ minimally.
\end{lem}

\section{Distinguishing the tangles}\label{nested_sets}
The aim in this section is to construct for any graph  a nested set of separations of finite order 
that 
distinguishes any two vertex-ends efficiently, which is needed in the proof of \autoref{undom_td}.  
A related result is proved in \cite{CDHS:can1}.
Actually, we shall prove the stronger statement that there is a 
nested set $\Ncal$ of separations that distinguishes any two tangles efficiently. 
A simplified version of this proof for finite graphs has been published in 
\cite{car_tangle-tree}.

{\bf Overview of the proof}

We shall construct the set $\Ncal$ as an ascending union of sets $\Ncal_k$ one for each $k\in \Nbb$, 
where $\Ncal_k$ is a nested set of separations of order at most $k$ distinguishing efficiently any 
two tangles\footnote{Actually this is not quite correct as we need `robust profiles' instead of 
`tangles'. This detail will be discussed at the end of the sketch. } of order $k+1$, see 
\autoref{fig:nested_tds}.
\begin{figure}
\begin{center}
   	  \includegraphics[height=3.5cm]{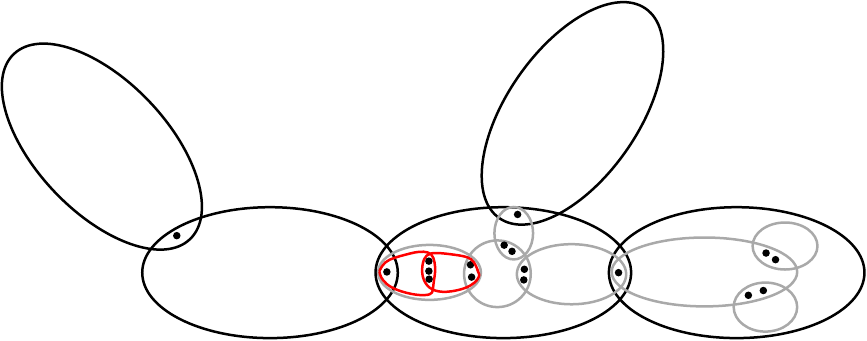}
   	  \caption{The tree-decomposition corresponding to the nested set $\Ncal_1$ is indicated by 
black parts. In torsos of that tree-decomposition, in grey we indicated a 
tree-decomposition for $\Ncal_2$. In each torso of that tree-decomposition we have a further 
tree-decomposition given by $\Ncal_3$, etc. }\label{fig:nested_tds}
\end{center}
   \end{figure}
Any two tangles of order $k+2$ that are not distinguished by $\Ncal_k$ will live 
in the same $\Ncal_k$-block.
We obtain $\Ncal_{k+1}$ from $\Ncal_k$ by adding for each $\Ncal_k$-block $\beta$ a nested set 
$\widetilde{ \Ncal_{k+1}(\beta)}$ that distinguishes efficiently any two tangles of order 
$k+2$ living in $\beta$. Working in the torsos $G_T[\beta]$ will ensure that the sets 
$\widetilde{ \Ncal_{k+1}(\beta)}$ for different blocks $\beta$ will be nested with each other.

Summing up, we are left with the task of finding in these torso graphs $G_T[\beta]$ a nested set 
distinguishing  efficiently tangles of order $k+2$.
\autoref{pre_tight_version} deals with this problem if the torso $G_T[\beta]$ is `nice enough'.
In order to make all torso graphs nice enough, we first do an additional step in which we enlarge 
$\Ncal_k$ a little bit so that for the larger nested set the new torso graphs are the old ones with 
the junk cut off. The main lemma for this enlargement is \autoref{S(k,r)_nice}.

As explained in \autoref{why_profiles} and \autoref{not_block}\footnote{Indeed, the robust profile 
induced by any tangle in the torso need not be a tangle. }, for such a torso-approach to work we 
need to work within the superclass of robust profiles that includes all the tangles (instead of 
just the tangles).

Finishing the overview, we first state \autoref{pre_tight_version} and \autoref{S(k,r)_nice}  
and introduce the necessary definitions for that.

For any robust profile $P$ and $k\in \Nbb$, its \emph{restriction} $P_k$ to $k$ consists of those 
separations in $P$ that have order at most $k$. The order of $P_k$ is the minimum of $k+1$ and the 
order of $P$.
A \emph{(robust) profile set} is a set of robust profiles that is closed under restrictions.
Until the end of Subsection \ref{subsec:tight}, we fix a graph $G$, a number $k\in 
\Nbb\cup\{\infty\}$ and a profile set $\Pcal$.

A nested set $\Ncal$ of separations is \emph{extendable} (for $\Pcal$) if 
for any two (distinct) robust profiles in $\Pcal$ of the same order, there is some separation 
distinguishing these two robust profiles efficiently that is nested with $\Ncal$.

A separation is \emph{relevant} (for a number $k\in \Nbb\cup\{\infty\}$, a graph $G$ and a 
profile set $\Pcal$) if it has order at most $k$ and it distinguishes two robust profiles
in $\Pcal$ efficiently -- in particular, it has finite order. We 
denote the set of all relevant separations by $R(k,\Pcal,G)$. 

Given a separation $(A,B)$, a component $C$ of $G-A\cap B$ is \emph{degenerated} if its 
neighbourhood\footnote{Throughout, we denote the neighbourhood of a vertex set $C$ by $N(C)$. } 
$N(C)$ is a proper subset of the separator $A\cap B$, see \autoref{fig:degenerated}.
\begin{figure}
\begin{center}
   	  \includegraphics[height=3cm]{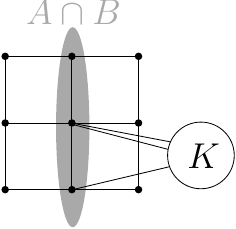}
   	  \caption{The separator $A\cap B$ is indicated in grey and has size three. The 
component $K$ of $G-A\cap B$ has only two neighbours in the separator 
$A\cap B$ and thus is degenerated. }\label{fig:degenerated}
\end{center}
   \end{figure}
 
A separation is \emph{degenerated} relative to $(A,B)$ if it is of the form 
$(C\cup N(C), G\sm C)$, where $C$ is a degenerated component of $G-A\cap B$. 
Given a set $\Scal$ of separations, its \emph{degenerator} is the set of separations that are 
degenerated relative to some separation in $\Scal$.
We denote the degenerator of the set $R(k,\Pcal,G)$ of relevant separations  by $S(k,\Pcal,G)$. 
If it is clear from the context what $G$ is, we shall just write $R(k,\Pcal)$ or $S(k,\Pcal)$, 
or even just $R(k)$ or $S(k)$.

\begin{eg}
Every relevant separation in $R(k)$ is tight if and only if $S(k)$ is empty.
\end{eg}

\begin{thm}\label{pre_tight_version}
Let $k\in \Nbb$. Assume that $S(k)=\emptyset$ and 
$R(k-1)=\emptyset$. Let $\Ncal$ be any nested subset of $R(k)$ that is inclusion-wise maximal.

Then $\Ncal$ distinguishes any two robust profiles of order $k+1$ in $\Pcal$ efficiently and is 
extendable.
\end{thm}

\begin{lem}\label{S(k,r)_nice}
If $R(k-1)$ is empty, then the degenerator $S(k)$ is a nested 
extendable set of separations.
\end{lem}

\subsection{Proof of \autoref{S(k,r)_nice}.}

In this subsection we prove \autoref{S(k,r)_nice}. First we need some preparation.

A separation $(A,B)$ \emph{pre-disqualifies} a separation $(C,D)$ if the order of $(C,D)$ is 
strictly larger than the sizes  $|L(A,C)|$ and $|L(B,C)|$ of corner separators. 
A separation $(A,B)$ \emph{disqualifies} a separation $(C,D)$ if it pre-disqualifies $(C,D)$ 
or its reverse $(D,C)$.

The following lemma shows that relevant separations cannot be disqualified. 

\begin{lem}\label{rem1}
If $(C,D)$ distinguishes robust profiles $P_1$ and $P_2$  efficiently, 
then no separation $(A,B)$ disqualifies $(C,D)$.
\end{lem}

\begin{proof}
We may assume that $C$ is big in $P_1$ and $D$ is big in $P_2$. 
Suppose for a contradiction that some separation $(A,B)$ pre-disqualifies $(C,D)$. 

So the order of $(C,D)$ is strictly larger than $|L(A,C)|$ and $|L(B,C)|$. The side $B\cup 
D$ of the corner separation $(A\cap C, B\cup D)$ is big in the robust profile $P_2$ as it includes a 
big 
side. By the efficiency of $(C,D)$, this corner separation cannot distinguish $P_1$ and $P_2$. Thus 
$A\cap C$ is small in $P_1$. A similar argument shows that also the corner $B\cap C$ must be small 
in $P_1$. This violates the robustness of $P_1$.  This is a 
contradiction to the assumption that $P_1$ is a robust profile. Hence $(A,B)$ cannot 
pre-disqualify $(C,D)$. 
Analogously, one shows that $(A,B)$ cannot pre-disqualify $(D,C)$. 
\end{proof}

\begin{lem}\label{S(k,r)_nice_calculate}
Let $(A,B)$ and $(C,D)$ be two separations distinguishing robust profiles in $\Pcal$ efficiently 
such that 
the order of $(A,B)$ is $k$ and the order of $(C,D)$ is at least $k$.
Let $K$ be a degenerated component of $G-A\cap B$.

If $R(k-1)$ is empty, then $K$ does not intersect the separator $C\cap D$.
\end{lem}

\begin{proof}
By symmetry, we may assume that the component $K$ is included in $A\sm B$. 

\begin{sublem}\label{x1}
 If $R(k-1)$ is empty, then the side $K\cup N(K)$ is small in every robust profile of order greater 
than $k$ of $G$. 
\end{sublem}

\begin{proof}
By assumption, there is a robust profile $P$ of order greater than $k$ such that the side $B$ is big 
in $P$.
As the side $G\sm K$ of the separation $(K\cup N(K),G\sm K)$ includes the big side $B$, it must 
also be big in $P$. Since $R(k-1)$ is empty, the side $K\cup N(K)$ is small in every robust profile 
of order 
greater 
than $k$. 
\end{proof}

\begin{sublem}\label{x2}
If the corner separation $(A\cap C, B\cup D)$ distinguishes two robust profiles  of order greater 
than $k$ efficiently, then the component $K$ does not intersect the separator $C\cap D$. 
\end{sublem}

\begin{proof}
The corner separation of the separations $(A\cap C, B\cup D)$ and $(G\sm K, N(K)\cup K)$ is 
$((A\cap C)\sm K, B\cup D\cup K)$. In particular, $((A\cap C)\sm K, B\cup D\cup K)$ is a 
separation. 
The separator of the separation  $((A\cap C)\sm K, B\cup D\cup K)$ is the corner separator $L(A,C)$ 
without $K$; in formulas $L(A,C)\sm K$. 
This separation also distinguishes two robust profiles efficiently: if the side $B\cup D$ is big in 
a 
robust profile, then also the side $B\cup D\cup K$ is big in that robust profile. By \autoref{x1} if 
the side 
$A\cap C$ is big in a 
robust profile, then also the side $(A\cap C)\sm K$ is big in that robust profile as it satisfies 
the component 
property and $G\sm K$ is big in all robust profiles. 
Hence by the efficiency of the separation $(A\cap C, B\cup D)$, it must be that the order of the 
separation $((A\cap C)\sm K, B\cup D\cup K)$ is at least the order of $(A\cap C, B\cup D)$; 
that is, the corner separator $L(A,C)$ does not intersect the component $K$.

If the component $K$ intersects the  separator $C\cap D$, 
it does so in the link $(C\cap D)\sm B$ as $K$ is a subset of $A\sm B$. Since this link is a subset 
of the corner separator  $L(A,C)$, the component $K$ cannot intersect that link as shown above.  So 
the component $K$ does not intersect the separator $C\cap D$. 
\end{proof}

 Let $Q_1$ and $Q_2$ be two robust profiles distinguished efficiently by the separation $(C,D)$ such 
that 
$C$ is big in $Q_1$ and $D$ is big in $Q_2$. 
 By replacing the separation $(A,B)$ by the separation $(B\cup K, A\sm K)$ if necessary we may 
assume that the side $A$ is big in $Q_1$.

\begin{sublem}\label{corner_analyse}
Either $|L(A,C)|\leq |C\cap D|$ and the corner $A\cap C$ is big in $Q_1$ or else 
$|L(A,D)|\leq |C\cap D|$ and the corner $A\cap D$ is big in $Q_2$.
\end{sublem}

\begin{proof}
Either the side $A$ or $B$ must be big in the robust profile $Q_2$.
We distinguish two cases. 

{\bf Case 1: the side $B$ is big in $Q_2$.} 

If $|L(B,D)|< |A\cap B|$, then the corner $B\cap D$ is in $Q_2$ by the corner property. Thus the 
corner separation $(B\cap D, A\cup C)$ will distinguish $Q_1$ and $Q_2$, which is impossible by 
the efficiency of $(C,D)$. Thus  by \autoref{XXX} $|L(A,C)|\leq 
|C\cap D|$, yielding that the corner  $A\cap C$ is big in $Q_1$ by the corner 
property, as desired.

{\bf Case 2: the side $A$ is big in $Q_2$.} 
 
By \autoref{rem1}, the separation $(C,D)$ does not pre-disqualify $(B,A)$.
Thus either $|L(B,C)| \geq |A\cap B|$ or $|L(B,D)| \geq |A\cap B|$. 
In the first case,  by \autoref{XXX} $|L(A,D)| \leq |C\cap D|$.
Then the corner $A\cap D$ is big in $Q_2$ by the corner property.
Similarly in the second case, $|L(A,C)| \leq |C\cap D|$.
Then the corner $A\cap C$ is big in $Q_1$ by the corner property, as desired. 
\end{proof}
 
 By \autoref{corner_analyse}, one of the corner separations $(A\cap C, B\cup D)$ or $(A\cap D, 
B\cup C)$ distinguishes the robust profiles $Q_1$ and $Q_2$ efficiently. Hence by \autoref{x2} or 
the 
corresponding fact for the corner separation $(A\cap D, 
B\cup C)$, we deduce that the component $K$ does not intersect the separator $C\cap D$.

\end{proof}

\begin{proof}[Proof of \autoref{S(k,r)_nice}.]
Let $(A,B)$ be a relevant separation in $R(k)$ and $(C,D)$ be some separation that distinguishes 
two robust profiles efficiently of order at least $k$.
Let $K_1$ be a degenerated component of $G-A\cap B$ and ${K_2}$ be a component of $G-C\cap D$. 
In order to see that $S(k)$ is a nested, it suffices to show that for any such $K_1$ and ${K_2}$ 
that the separations $(K_1\cup N(K_1), G\sm K_1)$ and $(K_2\cup N(K_2), G\sm K_2)$ are nested.
This is true by \autoref{S(k,r)_nice_calculate} and \autoref{are_nested}.
In order to see that $S(k)$ is an extendable, it suffices to show that for any such $K_1$ and 
$(C,D)$ 
that the separations $(K_1\cup N(K_1), G\sm K_1)$ and $(C,D)$ are nested.
This is true by \autoref{S(k,r)_nice_calculate} and \autoref{are_nested}, as well.
\end{proof}

\subsection{Proof of \autoref{pre_tight_version}.}\label{subsec:tight}

We actually prove the following extension of \autoref{pre_tight_version}. It is more general in the 
sense that it allows for even more flexibility which sets $\Ncal$ we could choose.
Recall that a separation is relevant (in $R(\infty)$) if it distinguishes some two robust profiles 
efficiently.

\begin{thm}\label{tight_version}
Let $k\in \Nbb$. Assume that $S(k)=\emptyset$ and 
$R(k-1)=\emptyset$. 
Any set $\Ncal$ of nested tight separations of order at most $k$ that are not disqualified by any 
relevant separation is extendable.

In particular, any maximal such set distinguishes any two robust profiles of order $k+1$ in $\Pcal$ 
efficiently. 
\end{thm}

Before we prove \autoref{tight_version}, we need some intermediate lemmas.
Throughout this subsection, we assume that $S(k)$ is empty.
Let $U$ be the set of those tight separations of order at most $k$ that are not disqualified by any 
relevant separation. Since $R(k)$ is a subset of $U$, \autoref{tight_version} implies 
\autoref{pre_tight_version}.

\begin{lem}\label{not_nested_only_finitely_many}
For any relevant separation $(A,B)$ such that $A\sm B$ is connected, there are only finitely many 
separation $(C,D)\in U$ not nested with $(A,B)$.
\end{lem}

\begin{proof}
First, we show that the separation $(A,B)$ is nested with every separation $(C,D)\in U$ such that 
the link $(A\cap B)\sm C$ 
is 
empty.
By \autoref{are_nested}, it suffices to show that the link $(C\cap D)\sm B$ is empty.
As $(A,B)$ does not pre-disqualify $(D,C)$, one of the links $(C\cap D)\sm A$ or 
$(C\cap D)\sm B$ is empty.
As we are done otherwise, we may assume that the link $(C\cap D)\sm A$ is empty. 
If $(C,D)$ is not nested with $(A,B)$, there must be a component of $K$ of $G-(C\cap D)$ all of 
whose 
neighbours are in the center $(A\cap B)\cap(C\cap D)$.
As $(C,D)$ is tight, it must be that $(C\cap D)=(A\cap B)\cap(C\cap D)$ so that $(C\cap D)\sm 
B$ is empty.
Hence $(A,B)$ and $(C,D)$ are nested by \autoref{are_nested}.

Similarly one shows that the separation $(A,B)$ is nested with every separation $(C,D)\in U$ such 
that the link $(A\cap B)\sm 
D$ is empty.

It remains to show that there are only finitely many separations $(C,D)\in U$ not nested with 
$(A,B)$. As 
shown above, in that case both 
links $(A\cap B)\sm C$ and $(A\cap B)\sm D$ are nonempty.
By \autoref{dk}, there are only finitely many triples $(v,w,T)$ where $v,w\in (A\cap B)$ and $T$ 
is a separator of size at most $k$ separating $v$ and $w$ minimally.
Since each separator $C\cap D$ for some $(C,D)$ as above is such a separator $T$, it suffices to 
show that 
there are only finitely many separations in $U$ that have the same separator as $(C,D)$.
This is true as the connected\footnote{Recall that the assumption that $S(k-1)$ is empty 
implies that the graph $G$ is connected. } graph $G$ without the separator $C\cap D$ has only 
finitely many components by 
\autoref{many_comp_nested} (in fact it has at most $|C\cap D|+1$ components).
\end{proof}

\begin{lem}\label{lem2}
Let $\Ncal$ be a nested subset of $U$. For any two robust profiles $P$ and $Q$ of order $\ell \geq 
k+1$ 
that are not 
distinguished by any separation of order less than $k$, there is some 
separation $(A,B)$ that is nested with $\Ncal$ 
and distinguishes $P$ and $Q$ 
efficiently. 
\end{lem}

\begin{proof}
First, we show that there is a separation $(A,B)$ distinguishing $P$ and $Q$ 
efficiently that is nested with all but finitely many separations of $\Ncal$.
Since $S(k)$ is empty, $R(k)$ is a subset of $U$. Let $(A,B)$ be a separation distinguishing 
$P$ and $Q$ efficiently.
As the robust profiles $P$ and $Q$ have the component property, we can pick (and we do pick) the 
separation 
$(A,B)$ 
such that $A\sm B$ is 
connected. By \autoref{not_nested_only_finitely_many}, 
$(A,B)$ is nested with all but finitely many separations of $\Ncal$.
Hence we can pick a separation $(A,B)$ distinguishing $P$ and $Q$ efficiently such that it is not 
nested with a 
minimal number of $(C,D)\in \Ncal$. 

Suppose for a contradiction that there is some separation
 $(C,D)\in \Ncal$ that is not nested with $(A,B)$. 
We may assume that $(C,D)$ does not distinguish $P$ and $Q$ since otherwise $(C,D)$ would 
distinguish $P$ 
and $Q$ efficiently by assumption.
Thus either the side $C$ is big in both $P$ and $Q$ or else the side $D$ is big in both $P$ and $Q$.
Since $(D,C)$ is nested with $\Ncal$, we may by symmetry assume that $C$ is big in both $P$ and 
$Q$. 

Since $(A,B)$ does not pre-disqualify $(D,C)$ by the definition of $U$, either 
$|L(A,D)|\geq 
|C\cap D|$ or $|L(B, D)|\geq |C\cap D|$.
By symmetry, we may assume that $|L(A,D)|\geq |C\cap D|$. 
By exchanging the roles of $P$ and $Q$ if necessary, we may assume that $A$ is big in $P$ and 
$B$ is big in $Q$. 
By \autoref{XXX}, $|L(B, C)|\leq |A\cap B|$. Note that the corner $B\cap C$ is small in $P$ 
as it is included in the small side $B$ of $P$. On the other hand, by the corner property the 
corner  $B\cap C$ is be big in $Q$.
Thus the corner separation $(B\cap C, A\cup D)$ distinguishes $P$ and $Q$ efficiently. 
Any separation in $\Ncal$ not nested with the corner separation $(B\cap C, A\cup D)$ is 
by \autoref{corner_nested} not nested with 
$(A,B)$.
As $(C,D)$ is nested with the corner separation $(B\cap C, A\cup D)$, this corner separation 
 violates the minimality of $(A,B)$.
Hence $(A,B)$ is nested with $\Ncal$, completing the proof.
\end{proof}

\begin{proof}[Proof of \autoref{tight_version}.]
By \autoref{lem2} and since $R(k-1)$ is empty, any nested subset $\Ncal$ of $U$ is extendable.

Since by assumption any relevant separation in $R(k)$ is in the set $U$, it follows that any 
maximal such set $\Ncal$ distinguishes any two robust profiles of order $k+1$ in 
$\Pcal$. It distinguishes efficiently as $R(k-1)$ is empty.  
\end{proof}

\subsection{Proof of the main result of this section.}

In this subsection, we prove the following.

\begin{thm}\label{nested_set_exists}
 For any graph $G$, there is a nested set $\Ncal$ of separations 
that distinguishes efficiently any two robust profiles (that are not restrictions of one another).
\end{thm}

First we need an intermediate lemma about sticking together a nested set $\Ncal$ of 
proper separations 
with nested sets of separations in the torsos of the $\Ncal$-blocks. 
We fix a finite number $k$ and a profile set $\Pcal$.
Let $\Ncal$ be a nested set of separations of order at most $k$ that is extendable for $\Pcal$
and that distinguishes  efficiently any two robust profiles of $\Pcal$ that can be distinguished by 
a 
separation of order at most $k$ in $G$.
For each $\Ncal$-block $\beta$, we denote by $\Pcal(\beta)$ the set of robust profiles in $\Pcal$ 
living in 
$\beta$.
And  let $\Ncal_\beta$ be a set of nested separations of the torso $G_T[\beta]$ of $\beta$ that is 
extendable for the induced robust profiles, induced by those robust profiles in $\Pcal(\beta)$.
We abbreviate $\Mcal=\Ncal\cup \bigcup \widetilde {\Ncal_\beta}$, where the union ranges over all 
$\Ncal$-blocks $\beta$. (Here in order to define the sets $\widetilde{\Ncal_\beta}$ we choose 
arbitrary well-orderings on the sets $\Ncal_\beta$.)

\begin{lem}\label{extendable_iterate}
The set $\Mcal$ is nested, proper and extendable for $\Pcal$.
\end{lem}

\begin{proof}
The set $\Mcal$ is nested by \autoref{about_different_blocks}. The separations in $\Ncal$ are 
proper by assumption and those in some $\Ncal_\beta$ are proper as they are efficient. By 
\autoref{same_order} extensions of proper separations are proper. 

It remains to show for every $\ell\geq k+1$ and any two robust profiles $P$ and $Q$ in $\Pcal$ that 
are 
distinguished efficiently by a separation of order $\ell$ in $G$ that there is a separation nested 
with 
$\Mcal$ that distinguishes $P$ and $Q$ efficiently. 
We may assume that $P$ and $Q$ both have order $\ell+1$ as $\Pcal$ is a profile set. 
Since  $\Ncal$ is extendable, there is a separation $(A,B)$ of order $\ell$ nested with $\Ncal$ 
that distinguishes $P$ and $Q$. By \autoref{extendable_lemma_rewritten_3}, there is a unique 
$\Ncal$-block $\beta$ including the separator $A\cap B$ such that $P$ and $Q$ live in $\beta$. 
The restriction of $(A,B)$ to $\beta$ distinguishes the robust profiles $P_\beta$ and $Q_\beta$, 
induced by 
$P$ and $Q$ respectively. As $\Ncal_\beta$ is extendable, there is a separation $(A',B')$ of the 
torso $G_T[\beta]$ that distinguishes $P_\beta$ and $Q_\beta$ efficiently; in particular it has 
order at most $\ell$. By \autoref{about_different_blocks}, the extension $\widetilde 
{(A',B')} $ is nested with $\Mcal$. By  \autoref{same_order} it has order at most $\ell$. 
So by \autoref{from_torso_to_G} it distinguishes $P$ and $Q$ efficiently. 
As $P$ and $Q$ were arbitrary, the nested set $\Mcal$ is extendable. 
\end{proof}

\begin{proof}[Proof of \autoref{nested_set_exists}.]
We shall construct the nested set $\Ncal$ of \autoref{nested_set_exists} as a nested union of sets 
$\Ncal_k$ one for each $k\in \Nbb\cup\{-1\}$, where
$\Ncal_k$ is a nested extendable set of separations of order at most $k$ that distinguishes any two 
robust profiles efficiently that are distinguished by a separation of order at most $k$. 
We start the construction with $\Ncal_{-1}=\emptyset$. 
Assume that we already constructed $\Ncal_k$ with the above properties.

 We denote by $\Pcal$ the set of all robust profiles of $G$. 
For an $\Ncal_k$-block $\beta$, we denote the set of robust profiles in $\Pcal$ living in $\beta$ 
by 
$\Pcal(\beta)$, and by $\Pcal_\beta$ the induced robust profiles of $\beta$,  induced by robust 
profiles in 
$\Pcal(\beta)$. Note that  $\Pcal_\beta$ is a profile set by \autoref{induces_in_torso}. 

\begin{sublem}\label{r_empty}
 The set $R(k,\Pcal_\beta, G_T[\beta])$ is empty. 
\end{sublem}

\begin{proof}
Suppose for a contradiction, two robust profiles $P_\beta$ and $Q_\beta$ in $\Pcal_\beta$ can be 
distinguished by a 
separation $(A,B)$ of order at most $k$. Then $\widetilde{(A,B)}$ has the same order as 
$(A,B)$ by \autoref{same_order}. It distinguishes the robust profiles $P$ and $Q$ which induce 
$P_\beta$ and 
$Q_\beta$ by \autoref{from_torso_to_G}. Since 
$P_\beta$ and 
$Q_\beta$ are distinct, also $P$ and $Q$ are distinct. But then  by the induction 
hypothesis $P$ and $Q$ are distinguished by $\Ncal_k$ -- as they are distinguishable by the 
separation $\widetilde{(A,B)}$ of order at most $k$. This 
contradicts the fact that $P$ and $Q$ are both in $\Pcal(\beta)$.
\end{proof}

By \autoref{r_empty}, we can apply \autoref{S(k,r)_nice} to the torso graph $G_T[\beta]$ and 
$\Pcal_\beta$, yielding that the degenerator $S(k+1,\Pcal_\beta, G_T[\beta])$ is a nested 
extendable set 
of 
separations.
For each $S(k+1,\Pcal_\beta, G_T[\beta])$-block $\beta'$, we define $\Pcal(\beta')$ and  
$\Pcal_{\beta'}$ similarly as $\Pcal(\beta)$ and $\Pcal_{\beta}$, respectively. 

\begin{sublem}\label{s_empty}
 The degenerator $S(k+1,\Pcal_\beta',G_T[\beta'])$ is empty. 
\end{sublem}

\begin{proof}
Suppose for a contradiction that this degenerator is not empty. Then there is a relevant 
separation $(C,D)$ in $R(k+1,\Pcal_\beta', 
G_T[\beta'])$
that has a degenerated component. By \autoref{same_order} and \autoref{from_torso_to_G}, the 
extension $\widetilde{ (C,D)}$ of $(C,D)$ distinguishes efficiently two robust profiles in 
$\Pcal(\beta)$. In particular that extension is relevant in $G_T[\beta]$, that is, it is 
contained in $R(k+1,\Pcal_\beta,G_T[\beta])$.

By assumption there is a degenerated component $K'$ of $\beta'$ without the separator $C\cap D$. 
The component $K'$ is included in a component $K$ of $\beta$ without the 
separator $C\cap D$. As $K'$ and $K$ have the same neighbours in the separator $C\cap D$, also 
$K$ is degenerated. So the separation $(K\cup N(K), \beta\sm K)$ is in the degenerator 
$S(k+1,\Pcal_\beta, G_T[\beta])$. Thus $\beta'$ is disjoint from the component $K$. So $K'$ is 
empty, which is the desired contradiction. 
\end{proof}

By Zorn's Lemma we pick a maximal nested subset $\Ncal(\beta')$ of $R(k+1, \Pcal_{\beta'}, 
G_T[\beta'])$, that 
is, of separations of order at most $k+1$ in the graph $G_T[\beta']$ distinguishing 
efficiently two robust profiles in $\Pcal_{\beta'}$. 
By \autoref{pre_tight_version} the set $\Ncal(\beta')$ is extendable for $\Pcal_{\beta'}$ and 
distinguishes 
any two robust profiles of order $k+2$ in $\Pcal_{\beta'}$ efficiently. 

Let $\Ncal_{k+1}(\beta)$ be the union of the nested set $S(k+1,\Pcal_\beta,G_T[\beta])$ with the 
sets 
$\widetilde{\Ncal(\beta')}$, where $\beta'$ is an 
$S(k+1,\Pcal_\beta,G_T[\beta])$-block.
By \autoref{extendable_iterate}, $\Ncal_{k+1}(\beta)$ is a nested and extendable set of separations 
of 
order at most $k+1$ in $G_T[\beta]$.
Let $\Ncal_{k+1}$ be the union of $\Ncal_{k}$ with the sets $\widetilde{\Ncal_{k+1}(\beta)}$, where 
$\beta$ 
is an $\Ncal_{k}$-block.
By applying \autoref{extendable_iterate} again, we deduce that $\Ncal_{k+1}$ is a nested and 
extendable set of separations of order at most $k+1$ in $G$.

\begin{sublem}\label{check_eff}
$\Ncal_{k+1}$ distinguishes efficiently any two robust profiles $P$ and $Q$ of $G$ that are 
distinguished by a separation of order at most $k+1$.
\end{sublem}

\begin{proof}
As $\Ncal_k$ is a subset of $\Ncal_{k+1}$, we may assume  by 
the induction hypothesis that any separation distinguishing $P$ and $Q$ efficiently 
has order $k+1$. Let $(C,D)$ be such a 
separation distinguishing them efficiently. As $\Ncal_{k+1}$ is extendable by 
\autoref{extendable_iterate}, we can pick (and we do pick) $(C,D)$ so that it is nested with 
$\Ncal_{k+1}$.

By \autoref{extendable_lemma_rewritten_3}, there is an $\Ncal_{k}$-block $\beta$ including the 
separator $C \cap D$ such that $P$ and $Q$ live in $\beta$. The induced robust profiles of $P$ and 
$Q$ in 
$\beta$ are denoted by $P_\beta$ and $Q_\beta$, respectively. The restriction of $(C,D)$ is nested 
with the degenerator $S(k+1,\Pcal_\beta, 
G_T[\beta])$ by construction. 
By \autoref{extendable_lemma_rewritten_3}, there is an $S(k+1,\Pcal_\beta, 
G_T[\beta])$-block $\beta'$ including 
the separator $C\cap D$ such that $P$ and $Q$ induce distinct robust profiles in $\Pcal_{\beta'}$. 
These 
induced robust profiles are distinguished efficiently by $\Ncal(\beta')$ by construction. Applying 
\autoref{from_torso_to_G} twice and \autoref{same_order} yields that $P$ and $Q$ are distinguished 
by $\Ncal_{k+1}$. As every separation in $\Ncal_{k+1}$ has order at most $k+1$, the robust profiles 
$P$ and 
$Q$ are distinguished efficiently by $\Ncal_{k+1}$. 
\end{proof}

Finally, the nested union $\Ncal$ of the sets $\Ncal_{k}$ is a nested set of separations that 
distinguishes efficiently any two robust profiles of the same order, as desired.
\end{proof}

\begin{cor}\label{thm1}
 For any graph $G$, there is a nested set  $\Ncal$ of finite separations that contains for any 
two vertex-ends a separation distinguishing them efficiently.
\end{cor}
\begin{proof}
By \autoref{end_is_tangle}, each vertex-end induces a tangle, which in return defines a robust 
profile. All these tangles are distinct (also as robust profiles) for 
different vertex-ends. So this is a consequence of 
\autoref{nested_set_exists}.
\end{proof}

\section{A tree-decomposition distinguishing the topological ends}\label{sec:undom_td}

In this section, we prove \autoref{undom_td} already mentioned in the introduction.
A key lemma in the proof of \autoref{undom_td} is the following.

\begin{lem}\label{star_td}
Let $G$ be a graph with a finite nonempty set $W$ of vertices.
Then $G$ has a star-decomposition\footnote{A \emph{star-decomposition is a tree-decomposition, 
where the decomposition tree is a star. }} $(S,Q_s|s\in V(S))$ of finite adhesion
such that each topological end lives in a part $Q_s$ with $s$ a leaf.

Moreover, only the central part $Q_c$ contains vertices of $W$, 
and for each leaf $s$, a topological end lives in the part $Q_s$, and the set $Q_s\sm Q_c$ is 
connected.
\end{lem}

\begin{proof}[Proof that \autoref{star_td} implies \autoref{undom_td}.]
We shall recursively construct a sequence  $\Tcal^n=(T^n,P_t^n|t\in V(T^n))$ of tree-decompositions 
of $G$ of finite adhesion as follows.
We start by picking a vertex $r'$ of $G$ arbitrarily and we obtain $\Tcal^1$ by applying 
\autoref{star_td} with $W=\{r'\}$. We refer to $r'$ as the \emph{rooting vertex}.
Assume that we already constructed $\Tcal^n$. For each leaf $s$ of $\Tcal^n$, 
we denote by $W_s$ the set of those vertices in $Q_s$ also contained in some other part of 
$\Tcal^n$.
Note that $W_s$ is contained in the part adjacent to $Q_s$ and thus is finite.
By \autoref{star_td}, we obtain a star-decomposition $\Tcal_s$ of $G[Q_s]$ 
such that no $w\in W_s$ is contained in a leaf part of $\Tcal_s$ and such that each 
topological end living in $Q_s$ lives in a leaf of $\Tcal_s$.
We obtain $\Tcal^{n+1}$ from $\Tcal^n$ by replacing each leaf part $Q_s$ by $\Tcal_s$, which is 
well-defined as the set $W_s$ is contained in a unique part of $\Tcal_s$.

By $r$, we denote the center of $\Tcal_1$.
For each $j<m<n$, the balls of radius $j$ around $r$ in $T^m$ and $T^n$ are the same.
Thus we take $T$ to be the tree whose nodes are those that are eventually a node of $T^n$.
For each node $t\in V(T)$, the parts $P_t^n$ are the same for $n$ larger than the distance between 
$t$ and $r$, and we take $P_t$ to be the limit of the $P_t^n$.

It is easily proved by induction that each vertex in the set  $W_s$ for $s$ a leaf of the tree 
$T^n$ has distance at least $n-1$ from the rooting vertex $r'$ in the graph $G$.
Thus for each $j<n$ the ball of radius $j$ around the rooting vertex $r'$ in $G$ is included in the 
union over all parts $P_t^n$ where $t$ is in the ball of radius $j$ around $r$ in $\Tcal_n$.
Hence $(T,P_t|t\in V(T))$ is a tree-decomposition, and it has finite adhesion by construction.

It remains to show that the ends of $T$ define precisely the topological ends of $G$, which is done 
in the following four sublemmas. 

\begin{sublem}\label{undom_in_tau}
 Each topological end $\omega$ of $G$ lives in an end of $T$.
\end{sublem}

\begin{proof}
There is a unique leaf $s$ of $T^n$ such that $\omega$ lives in $P_s^n$. Let $s_n$ be the 
predecessor of $s$ in $T^n$.
Then $\omega$ lives in the end of $T$ to which $s_1s_2\ldots$ belongs.
\end{proof}

\begin{sublem}\label{end_in_tau}
 In each end $\tau$ of $T$, there lives a vertex-end of $G$. 
\end{sublem}
\begin{proof}
Let $X$ be a spanning tree of the graph $G$. Our aim is to find a ray included in $X$ whose 
vertex-end lives in the end $\tau$.

Let $s_1s_2...$ be the ray in $T$ starting at $r$ that belongs to the end $\tau$. By construction, 
the sets $W_{s_i}$ are disjoint and finite. Let $U$ be the union of the sets $W_{s_i}$. Since each 
vertex is separated by some set $W_{s_i}$ from all but finitely many vertices of $U$, the tree $X$ 
does not include a subdivision of an infinite star with all leaves in $U$. Hence by the 
Star-Comb-Lemma\footnote{The Star-Comb-Lemma says that if $U$ is an infinite vertex set in 
a tree $X$, then either $X$ contains a subdivision of an infinite star with all leaves in $U$ or 
$X$ contains a comb with all leaves in $U$; here a comb is obtained from a ray by attaching a path
at each vertex. } \cite[Section 8]{DiestelBook10}, there is a comb with infinitely many leaves in 
the set $U$. Thus the vertex end of the ray of that comb lives in the end $\tau$. 
\end{proof}

\begin{sublem}\label{unique_tau}
 No two distinct vertex-ends $\omega_1$ and $\omega_2$ of $G$ live in the same end $\tau$ of $T$.
\end{sublem}

\begin{proof}
Suppose for a contradiction, there are such vertex-ends $\omega_1$, $\omega_2$ living in the same 
end $\tau$.
Let $U$ be a finite separator separating  $\omega_1$ from $\omega_2$ and let $n$ be the 
maximum over the distances between the rooting  vertex $r'$ and a vertex in $U$.
Let $s$ be the unique node of the tree T on the ray starting at the root belonging to the end $\tau$ 
that has distance $n$ from the root. By construction, in the tree $T^{n+1}$ the node $s$ is leaf. 
Let $C_i$ be the component of the graph $G-U$ in which the vertex-end $\omega_i$ lives. 
Recall that the leaf-part $Q_s$ (of $\Tcal_{n+1}$) with the separator $W_s$ removed is connected. 
Since 
the set $W_s$ separates the separator  $U$ from the set $Q_s\sm W_s$, the connected 
set $Q_s\sm W_s$ is contained in a component of the graph $G-U$. As the vertex-end $\omega_i$ 
lives in the 
graph $Q_s\sm W_s$ by assumption, it must be that the set $Q_s\sm W_s$ is a subset of the 
component $C_i$.
Hence the components $C_1$ and $C_2$ intersect, which is the desired contradiction.
\end{proof}

\begin{sublem}\label{undom_omega}
 No vertex $u$ dominates a vertex-end $\omega$ living in some end of $T$.
\end{sublem}

\begin{proof}Suppose for a contradiction a vertex $u$ dominates the vertex-end $\omega$.
Let $n$ be the distance between the vertex $u$ and the rooting vertex $r'$ in $G$. Then there is a 
leaf $s$ of the tree $T^{n+1}$ such that the vertex-end $\omega$ lives in the leaf-part $Q_s$.
Thus the finite set $W_s$ separates the vertex $u$ from the vertex-end $\omega$, contradicting the 
assumption that the vertex $u$ dominates the vertex-end $\omega$.
\end{proof}

\autoref{undom_in_tau}, \autoref{end_in_tau}, \autoref{unique_tau} and \autoref{undom_omega} imply 
that the ends of $T$ define precisely the topological ends of $G$, as desired.
\end{proof}

\begin{rem}\label{rem:connectedness}
Let $(T,\leq)$ be the tree order on $T$ as in the proof of \autoref{undom_td} where the root $r$ is 
the smallest element.
We remark that we constructed $(T,\leq)$ such that $(T,P_t|t\in V(T))$ has the following additional 
property. 
For each edge $tu$ with $t\leq u$, the vertex set $\bigcup_{w\geq u} V(P_w)\sm V(P_t)$ is connected.

Moreover, we construct $(T,P_t|t\in V(T))$ such that if $st$ and $tu$ are edges of $T$ with $s\leq 
t\leq u$, then $V(P_s)\cap V(P_t)$ and $V(P_t)\cap V(P_u)$ are disjoint.
\end{rem}

In order to prove \autoref{star_td}, we need the following.

\begin{lem}\label{nested}
Let $G$ be a connected graph and $W$ a finite and nonempty vertex set of $G$.
Then there is a set $\Xcal$ of separations $(A_i,B_i)$ of finite order such that 
every vertex-end not dominated by a vertex of $W$ lives in a side $B_i$.
Moreover, the sets $B_i\sm A_i$ are disjoint and the set $W$ is vertex-disjoint from all sides 
$B_i$. 
\end{lem}

\begin{proof}[Proof that \autoref{nested} implies \autoref{star_td}.]
If a set $B_i\sm A_i$ has several components, we replace the separation $(A_i,B_i)$ in $\Xcal$ 
by the set of separations for the form $(G\sm C, C\cup N(C))$, where $C$ is a component of $B_i\sm 
A_i$. 
Call the resulting set $\Xcal'$. By replacing $\Xcal$ by $\Xcal'$ if necessary, we may assume that 
all sets $B_i\sm A_i$ in \autoref{nested} are connected. 

We may assume that the graph $G$ in \autoref{star_td} is connected.
Our aim is to construct a star-decomposition of $G$. It gets a leaf-part for every every separation 
$(A_i,B_i)$ in $\Xcal$ such that a topological end lives in $B_i$. Its part is $B_i$. 

Let $C$ be the intersection of the sides $A_i$ with $(A_i,B_i)\in \Xcal$ together with all sides 
$B_i$ such that no topological end lives in $B_i$. 
The set $C$ gets the central part of our star-decomposition. 
By construction, this is a star-decomposition and it has finite adhesion. 
This star-decomposition has all the desired properties by construction. 
\end{proof}

The rest of this section is devoted to the proof of \autoref{nested}.
We shall need the following lemma.

\begin{lem}\label{auxi1}
Let $G$ be a connected graph and $W$ a finite nonempty vertex set.
There is a nested set $\Ncal$ of proper separations $(A_i,B_i)$ of finite order such that 
every vertex-end not dominated by a vertex of $W$ lives in a side $B_i$ and the set $W$ is 
vertex-disjoint from all sides $B_i$. 

Moreover, if two distinct separations $(A_i,B_i)$ and $(A_j,B_j)$ of $\Ncal$ satisfy 
$(A_i,B_i)\leq(A_j,B_j)$, then the order of $(A_i,B_i)$ is strictly larger 
than the order of $(A_j,B_j)$.
\end{lem}

\begin{proof}
We obtain the graph $G_W$ from the graph $G$ by first deleting the vertex set $W$ and then adding a 
copy of  the complete graph\footnote{By $K_\omega$ we denote the complete graph on countably many 
vertices.} $K_\omega$ in such a way that it is joined completely to the neighbourhood of $W$ in $G$.
Applying \autoref{thm1} to the graph $G_W$, yields a nested set $\Ncal'$ of separations of 
finite order such 
that any two vertex-ends of $G_W$ are distinguished efficiently by a separation in $\Ncal'$.
Let $\tau$ be the vertex-end to which the rays of the newly added copy of $K_\omega$ belong.
Let $\Ncal''$ consist of those separations in $\Ncal'$ that distinguish $\tau$ efficiently from 
some other 
vertex-end. By reversing separations in $\Ncal''$ if necessary, we may assume that the added graph 
$K_\omega$ is included in the side $A$ for every separation $(A,B)\in \Ncal''$. 
As the separations in $\Ncal''$ distinguish efficiently, for no separation $(A,B)$ in $\Ncal''$ 
the side $B$ contains a vertex of the added graph $K_\omega$.

Given a natural number $k$, a \emph{$k$-sequence} $((A_\alpha, B_\alpha)|\alpha\in \gamma)$ 
(for $\Ncal''$) is an ordinal indexed sequence of elements of $\Ncal''$ of 
order at most $k$ such that if $\alpha<\beta$, then $B_\alpha\se B_\beta$.
(Recall that every separation in $\Ncal''$ is proper so $B_\alpha\se B_\beta$ implies that 
$(A_\beta,B_\beta)\leq(A_\alpha,B_\alpha)$ by \autoref{proper_nested}.) 
The \emph{union separation} of a $k$-sequence $((A_\alpha, B_\alpha)|\alpha\in \gamma)$ is the 
separation obtained by taking the union over all $B$-sides and the intersection of all $A$-sides, 
formally it is: $(\bigcap_{\alpha\in \gamma} A_\alpha, \bigcup_{\alpha\in \gamma} B_\alpha)$.

The set $\Ncal'''$ consists of all union separations of $k$-sequences of $\Ncal''$ 
for all $k$.
Since we allow constant sequences, the nested set $\Ncal''$ is included in the set $\Ncal'''$.
A standard transfinite induction argument yields that the set $\Ncal'''$ is 
nested.\footnote{Given two members of $\Ncal'''$, by replacing their underlying sequences by 
different sequences with the same union, one may assume that all members of their sequences are 
nested with each other in the same way (of the four ways separations could be nested). Hence by 
transfinite induction also their unions must be nested in that way.}
Given a natural number $k$, the set $\Ncal_k$ consists of those separations of the 
nested set $\Ncal'''$ that have order at most $k$; and the set $\Ncal_k'$ consists of those 
 elements $(A,B)$ of the nested set $\Ncal_k$ whose $B$-side is inclusion-wise maximal in $\Ncal_k$.

We take $\Ncal_W$ to be the union of the nested sets $\Ncal_k'$.
By construction, for each separation $(A,B)$ in $\Ncal_W$, the side $B$ contains no vertex of the 
 added graph $K_\omega$. We obtain $\Ncal$ from $\Ncal_W$ by replacing each separation $(A,B)$ in 
$\Ncal_W$ by the separation where we modify the side $A$ by replacing the added graph $K_\omega$ by 
the finite vertex set $W$. Clearly, the set $\Ncal$ is a nested set of proper\footnote{This 
properness follows from the fact that $(A,B)$ comes from a separation of $G_W$ distinguishing two 
tangles efficiently and the modifications made by going from $G_W$ to $G$ preserve being proper.} 
separations of the graph $G$. 

We claim that the nested set $\Ncal$ has all the properties stated in \autoref{auxi1}:
the `Moreover'-part is clear by construction.
Thus it remains to show that each vertex-end $\omega$ of $G$ not dominated by a vertex of $W$ 
lives in a side $B_i$ for some separation $(A_i,B_i)$ in the nested set $\Ncal$. 

Let $R$ be a ray belonging to the vertex-end $\omega$. Since the vertex-end $\omega$ is not 
dominated by any vertex of $W$, for 
each vertex $x$ of $W$ there is a finite vertex set $S_x$ separating a subray $R_x$ of $R$ from 
$x$. We let $S$ be the union of these finite separators $S_x$. The finite set $S$ separates the 
intersection $R'$ of the subrays $R_x$ from vertex set $W$ in the graph $G$. In the graph $G-W$, 
the set $S$ separates the ray $R'$ from the added graph $K_\omega$. 
Let $\omega'$ be the vertex-end of $G_W$ to which the ray $R'$ belongs. Note that the separator $S$ 
witnesses that the vertex-end $\omega'$ is not equal to the vertex-end $\tau$ of the added 
$K_\omega$. Thus there is a separation $(A_i,B_i)$ of $\Ncal'''$ so that the vertex-end $\omega'$ 
lives in the side $B_i$. Let $k$ be the order of the separation $(A_i,B_i)$. By Zorn's 
lemma, the nested $\Ncal'''$ contains a separation $(A',B')$ with $B_i\se B'$ of order at most $k$ 
whose $B$-side is inclusion-wise maximal amongst all separations of $\Ncal'''$ of order at most 
$k$. By construction the separation $(A',B')$ is in the nested set $\Ncal_k'$ and includes a subray 
of $R'$. So the separation $((A'\sm K_\omega)\cup W,B')$ is in $\Ncal$ and the vertex-end 
$\omega$ lives 
in the side $B'$. This completes the proof.\footnote{We sketch an alternative proof. First one 
shows that we may assume that every separation $(A,B)$ in $\Ncal''$ has the property that $A\sm B$ 
is connected. Then instead of referring to $\Ncal'''$ one can use the following fact: let $P$ be a 
tangle of order $k+1$ and let $(A_i,B_i)$ be a sequence of separations of order at 
most $k$ distinguishing tangles efficiently such that each $A_i$ is big in $P$ and the sets $A_i\sm 
B_i$ are connected and $A_{i+1}\se A_i$. Then the sequence has only finitely many distinct members.}
\end{proof}

Next we show how \autoref{auxi1} implies \autoref{nested}.
A good candidate for the nested set $\Xcal$ of \autoref{nested} might be the separations $(A,B)$ in 
the nested set $\Ncal$ such that the side $B$ is inclusion-wise maximal amongst members of the 
nested set $\Ncal$.
However, there might be an infinite strictly increasing -- that is, the $B$-sides are 
strictly increasing -- sequence of members in $\Ncal$, whose orders 
are also strictly increasing, so that we cannot expect that the union of all these $B$-sides is a 
side of a finite order separation, and hence cannot come from a separation in $\Ncal$, see 
\autoref{eg_canhappen}.
Thus we have to make a more sophisticated choice for $\Xcal$ than just taking the `maximal members' 
of $\Ncal$. 
\begin{eg}\label{eg_canhappen}
The set $W$ just consists of a single vertex, which is complete to a ray. At each initial path 
$P_n$ of the ray, we attach a copy of the ladder of width $n$. The set $\Ncal$ consists of those 
separations $(A_n,B_n)$ with separator $P_n$ separating the first $n$ ladders from the vertex of 
$W$. These separations are strictly increasing in order and this sequence does not have a 
maximal element in the sense explained above this example. 
\end{eg}

\begin{proof}[Proof that \autoref{auxi1} implies \autoref{nested}.]
Let $\Ncal$ be a nested set as in \autoref{auxi1}. The first step in the proof is to define a graph 
$H$ that visualises the structure of the nested set $\Ncal$. 

Let $(A,B)$ be a separation of the nested set $\Ncal$ such that there is another 
separation $(C,D)$ in $\Ncal$ with $B\se D$. (Recall that all separations of the set $\Ncal$ are 
proper. So this implies that $(C,D)\leq (A,B)$ by \autoref{proper_nested}.) 
Then the order of $(C,D)$ is larger than that of $(A,B)$. 

Such a separation $(C,D)$ is called a \emph{successor} of the 
side $B$ (we use the 
condensed notation `of $B$' instead of `of $(A,B)$' as every proper separation is uniquely 
determined by one of its sides); the separation $(C,D)$ is an \emph{immediate succesor} if it has 
minimal order amongst all successors. 
Let $H$ be the digraph with vertex set $\Ncal$
where we put in the directed edge $BD$ if the separation $(C,D)$ is an immediate successor of $B$. 
A \emph{connected component of $H$}, is a connected component of the underlying graph of $H$.
A typical example for a connected component of the graph $H$ is the canopy tree, see 
\autoref{canopy}. 

\begin{figure}
\begin{center}
   	  \includegraphics[height=2cm]{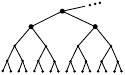}
   	  \caption{The canopy tree. We obtain an interesting example for the graph $H$ by 
directing all edges away from the leaves.}\label{canopy}
\end{center}
   \end{figure}

\begin{sublem}\label{dist_comp}
Let $(A',B')$ and $(C',D')$ be separations in $\Ncal$. Then $B'\se D'$ if and only if there is 
a directed path in $H$ from $B'$ to $D'$.
Moreover, if two separations $(A,B)$ and $(C,D)$ in $\Ncal$ are not 
joined by a directed path, then $B\sm A$ and $D\sm C$ are disjoint. 
\end{sublem}

\begin{proof}
Clearly, if there is a directed path from $B'$ to $D'$, then $B'\se D'$. Conversely, let $(A',B')$ 
and $(C',D')$ be separations in $\Ncal$ with $B'\se D'$. 
Let $(A_n,B_n)$ be a sequence of distinct separations in $\Ncal$ such that $B'\se B_1\se... \se 
B_n\se D'$. 
By \autoref{auxi1}, $n\leq |\partial(D')|-|\partial(B')|+1$. Thus there is a maximal such chain 
$(E_n,F_n)$, which satisfies
$F_1=B'$ and $F_n=D'$ and $F_{i+1}\in F(B_i)$ for all $i$ between $1$ and $n-1$. Hence $F_1...F_n$ 
is a path from $B'$ to $D'$.

To see that ``Moreover''-part, let  $(A,B)$ and $(C,D)$ be separations in $\Ncal$.
Since the set $W$ is nonempty, the side $B$ does not 
include the side $C$. As the sides $B$ and $D$ cannot be subsets of one another by 
assumption, the nestedness yields that  $B\se C$ and $D\se A$. Hence the sets $D\sm C$ 
and $B\sm A$ are disjoint. 
\end{proof}

\begin{sublem}\label{deg-lem}
Each vertex $v$ of $H$ has out-degree at most one.\footnote{We do not use it in our proof but it 
follows from this lemma that $H$ is a forest. Indeed, any cycle included in $H$ must by the 
outdegree condition be a directed cycle. This, however, is impossible by \autoref{dist_comp}.} 
\end{sublem}

\begin{proof}
Suppose for a contradiction the side $v$ has out-degree at least 2.
Then there are distinct immediate successors $(A,B)$ and $(C,D)$. By the conditions of 
\autoref{auxi1}, it must be that neither $B\se D$ 
nor $D\se B$.
Thus $B\sm A$ and $D\sm C$ are disjoint by \autoref{dist_comp}.
Since $v\se B\cap D$, it cannot be the side of a proper separation.
This is the desired contradiction to the assumption that $v$ is a side of a separation in $\Ncal$.
\end{proof}

\begin{sublem}\label{path}
 Any undirected path $P$ joining two vertices $v$ and $w$ contains a vertex $u$
such that $vPu$ and $wPu$ are directed paths which are directed towards $u$.
\end{sublem}

\begin{proof}
It suffices to show that the path $P$ contains at most one vertex of out-degree zero on $P$.
If it contained two such vertices, then between them would be a vertex of out-degree two, which 
is impossible by \autoref{deg-lem}.
\end{proof}

We define the set $\Xcal$ as the union of sets $\Xcal_K$, one for each component $K$ of $H$.
The sets $\Xcal_K$ are defined as follows.
If a component $K$ contains a vertex $v_K$ of out-degree $0$, 
then by \autoref{path} $K$ cannot contain a second such vertex and 
for any other vertex $v$ in the component $K$, there is a directed path from $v$ to the vertex 
$v_K$ directed towards $v_K$.
Hence the side $v_K$ includes any other the side $v$ that is a vertex of the component $K$. We 
choose for $\Xcal_K$ the unique separation in $\Ncal$ with the side $v_K$; here the uniqueness 
follows from the fact that proper separations are uniquely determined by one of their sides.

Otherwise, every vertex of the component $K$ has outdegree precisely one by \autoref{deg-lem}. 
Since the component $K$ cannot contain a directed cycle by 
\autoref{dist_comp}, it must contain a directed ray; that is, a ray $B_1B_2\ldots$  with 
$B_i \se B_{i+1}$. 
In this case, we define the set $\Xcal_K$ to consist of the separations $(C_i,D_i)$ defined as 
follows.
Let $(A_i,B_i)$ be the unique proper separation with side $B_i$.
We let $(C_1,D_1)=(A_1,B_1)$. Roughly, we obtain $(C_i,D_i)$ from $(A_i,B_i)$ by flipping the set 
$B_{i-1}\sm A_{i-1}$ from the side $B_i$ to the side $A_i$; in formulas for $i>1$, we let 
$C_i=A_i\cup (B_{i-1}\sm A_{i-1})$ and $D_i=B_i\sm (B_{i-1}\sm A_{i-1})$. 
Note that the order of $(C_i,D_i)$ is bounded by the sum of the orders of $(A_i,B_i)$ and 
$(A_{i-1},B_{i-1})$, and thus finite. 
Since no side $B_i$ contains a vertex of $W$, the same is true for the sides $D_i$.
This completes the definition of the set $\Xcal_K$ and thus $\Xcal$.

Any two distinct separations $(A,B)$ and $(C,D)$ in the set $\Xcal$ have the property that 
$B\sm A$ and $D\sm C$ are disjoint; indeed, if these separations are in the same set  
$\Xcal_K$, this is 
clear by construction.
Otherwise it follows from the definition of the side $D_i$ and \autoref{dist_comp}.
Thus it remains to prove the following:

\begin{sublem}\label{repre}
Every vertex-end $\omega$ not dominated by some vertex of $W$ lives in some side $B$ with $(A,B)\in 
\Xcal$.
\end{sublem}
\begin{proof}
By \autoref{auxi1}, there is a separation $(E,F)$ in the nested set $\Ncal$ such that the 
vertex-end $\omega$ lives in $F$. 
Let $K$ be the component of $H$ containing the vertex $F$.
If $\Xcal_K=\{v_K\}$, then $F\se v_K$; and we are done as the vertex-end $\omega$ lives in the side 
$v_K$.
Otherwise let the $B_i$ and the $D_i$ be as in the construction of $\Xcal_K$.
If $F=B_j$ for some $j$, then we pick $j$ minimal such that $\omega$ lives in $B_j$.
Since $\omega$ does not live in $B_{j-1}$, it must live in $D_j$, as desired.

Thus we may assume that the side $F$ is not equal to any side $B_j$.
Let $P$ be a path joining the vertex $F$ and the vertex $B_1=D_1$. 
By \autoref{path}, the path $P$ contains a vertex $u$
such that the subpaths $FPu$ and $B_1Pu$ are directed paths which are directed towards $u$.
Thus $F\se u$. Since the out-degree 
is at most one, the path $B_1Pu$ is a subpath of the ray $B_1B_2\ldots$.
Thus the vertex $u$ is equal to $B_j$ for some $j$. In particular, $F\se B_j$.

As the side $B_j$ is the union of the finitely many sides $D_1, D_2, ..., D_j$, the vertex-end 
$\omega$ has to live in some side $D_i$ with $i\leq j$. This completes the proof. 
\end{proof}
\end{proof}

Finally we deduce \autoref{Halin_type_intro}.

\begin{proof}[Proof that \autoref{undom_td} implies \autoref{Halin_type_intro}.]
By \autoref{undom_td}, $G$ has a tree-decomposition $(T,P_t|t\in V(T))$ of finite adhesion such that 
the ends of $T$ define precisely the topological ends of $T$, and we choose this tree-decomposition 
as in \autoref{rem:connectedness}. In particular, we can pick a root $r$ of $T$ such that for each 
edge $tu$ with $t\leq u$, the vertex set $\bigcup_{w\geq u} V(P_w)\sm V(P_t)$ is connected.

Furthermore for each edge $tu$ with $t\leq u$, one may assume that if a vertex is in the separator 
$V(P_t)\cap V(P_u)$, then it has a neighbour in $P_u\sm P_t$ by deleting other vertices from 
the part $P_u$ if necessary. 

Thus for each such edge $tu$, there is a finite connected subgraph $S_u$ of the 
induced subgraph $G[\bigcup_{w\geq u} V(P_w)]$ that contains the separator $V(P_t)\cap V(P_u)$. 
Let $Q_t$ be a maximal subforest of the union of the $S_u$, where the union ranges over all upper 
neighbours $u$ of $t$.
We recursively build a maximal subset $U$ of $V(T)$ such that if $a,b\in U$, then $Q_a$ and $Q_b$ 
are vertex-disjoint.
In this construction, we first add the nodes of $T$ with smaller distance from the root.
This ensures by the ``Moreover''-part of \autoref{rem:connectedness} that 
$U$ contains infinitely many nodes of each ray of $T$.

Let $S'$ be the union of those $Q_t$ with $t\in U$. 
We obtain $S$ by extending $S'$ to a spanning tree of $G$, and rooting it at some $v\in V(S)$ 
arbitrarily.
By the Star-Comb-Lemma \cite[Section 8]{DiestelBook10}, each spanning tree of $G$ contains for each 
topological end $\omega$ a ray belonging to $\omega$.

Thus it remains to show that $S$ does not contain two disjoint rays $R_1$ and $R_2$ that both belong 
to the same topological end $\omega$ of $G$.
Suppose there are such $R_1$, $R_2$ and $\omega$.
Let $t_1t_2\ldots$ be the ray of $T$ in which $\omega$ lives.
Let $n$ be so large that both $R_1$ and $R_2$ meet $P_{t_n}$. Then for each $m\geq n$, the set 
$S_{t_m}$ contains a path joining $R_1$ and $R_2$.
Thus the set $Q_{t_{m-1}}$ contains such a path. Since $Q_{t_{m-1}}\se S$ for infinitely many $m$, 
the tree $S$ contains a cycle, which is the desired contradiction.
\end{proof}

\begin{rem}\label{proof_more_general}
In the above proof of \autoref{Halin_type_intro} in the application of the Star-Comb-Lemma we used 
the property that topological ends are not dominated by vertices. 

However, with a little bit more care, one can show more generally that if a 
graph has a tree-decomposition such that the ends of the decomposition tree define 
precisely a set 
$\Psi$ of vertex-ends of the graph, then this graph has a 
spanning tree that is end-faithful for that set $\Psi$. 

To see that we show that if a graph $G$ has such a tree-decomposition $(T,P_t|t\in V(T))$, then it 
has a  connected subgraph $G'$ with the same vertex set such that all vertex-ends in the set 
$\Psi$ are vertex-ends of $G'$ that are topological. 

It is fairly easy to see that we may assume that the tree-decomposition $(T,P_t|t\in V(T))$ has the 
following additional properties.
\begin{enumerate} 
\item By $Q_t$ we denote the union of the part $P_t$ with all parts $P_u$, where 
$u$ is above $t$ 
in the 
decomposition tree, without the part $P_s$; here $s$ is the downward-neighbour of $t$ and $t$ 
is not the root. Then the graph $Q_t$ is connected.
 \item Every part $P_t$ contains a finite connected set $C_t$ such that $C_t\se P_t\sm P_s$ and 
every 
vertex of the separator $V(P_t)\cap V(P_s)$ is in the neighbourhood of $C_t$; here $s$ and $t$ are 
as in the first property.
 \item  Let $s\leq t\leq u$ such that $st,tu\in E(T)$. Then the separator $V(P_t)\cap V(P_u)$ is 
not a subset of $V(P_s)\cap V(P_t)$.
\end{enumerate}

For a node $t$ different from the root, let $K_t$ be the union of $P_t$ with all sets $C_u$, where 
$u$ is an 
upward-neighbour of $t$, without the part $P_s$, where $s$ is the down-ward neighbour of $t$.
Since the set $Q_t$  is connected and no separator $V(P_t)\cap V(P_u)$ is a subset of $P_s$ 
by the third property, the graph $K_t$ must be connected. If $t$ is the root, we define $K_t$ the 
same but without removing a part $P_s$; this graph is connected as the graph $G$ is connected. 

We define $G'$ to be the union of the connected subgraphs $K_t$. This graph is connected. It is 
straightforward to check that every vertex-end of $\Psi$ is a 
topological end of $G'$. So we can apply \autoref{Halin_type_intro} to the graph $G'$ to deduce 
that the graph $G$ has a spanning tree that is end-faithful for the set $\Psi$. 
\end{rem}

\section{Concluding Remarks}\label{concluding}

We have shown that any graph has a tree-decomposition of finite adhesion that distinguishes its 
topological ends. 
It is natural to ask whether such a statement is true if we replace `topological 
ends' by some other classes of vertex-ends. 

Let us be more precise. A class $\Ccal$ of vertex-ends 
is \emph{tree-distinguishable} if every graph has a tree-decomposition of finite adhesion that 
distinguishes any two vertex-ends that are in $\Ccal$. We would like to know which natural 
classes $\Ccal$ of vertex-ends are tree-distinguishable?

As demonstrated in \autoref{sketch_t2} the class of all vertex-ends is not tree-distinguishable.
A class that has received a lot of attention in the literature, see for example 
\cite{hp:transi}, is the class of 
`$k$-thin' vertex-ends; here, given a natural number $k$, a vertex-end $\omega$ is \emph{$k$-thin} 
if the number $k_1$ of vertices dominating $\omega$ and the cardinality $k_2$ of any family of 
vertex-disjoint rays belonging to $\omega$ sum up to at most $k$, that is $k_1+k_2\leq k$. 
For example, thin ends are $k$-thin for every sufficiently large value of $k$; indeed vertex-ends 
dominated by 
infinitely many vertices have infinitely many vertex-disjoint rays belonging to that end, see 
\cite{Diestelbookcurrent}. Although the class of thin vertex-ends is not tree-distinguishable by 
\autoref{thin_not_distinguished}, the class of $k$-thin vertex-ends is; this is well-known 
and also follows from \autoref{kalmosttopo} below.

For general graphs the class of $k$-thin 
vertex-ends and topological ends are not subsets of one another. Is there a natural 
tree-distinguishable class that contains both of them?

Yes, there is such a class and the proof that it is tree-distinguishable is an easy application of 
our main theorem, as follows.
Given a natural number $k$, a vertex-end is \emph{$k$-dominated} if it is dominated by at 
most $k$ vertices.
For example, the $0$-dominated vertex-ends are the topological ends.   
Clearly every $k$-thin vertex-end is $k$-dominated.
The following extension of \autoref{undom_td} implies that the class of $k$-dominated vertex-ends 
is 
tree-distinguishable for any fixed $k$.

\begin{thm}\label{kalmosttopo}
For any fixed natural number $k$, every graph has a tree-decomposition $(T,\Vcal)$ of finite 
adhesion such that the ends of $T$ 
define precisely the $k$-dominated vertex-ends of $G$. 
\end{thm}

\begin{proof}
We shall prove \autoref{kalmosttopo} by induction on $k$. All trees in this proof are rooted; and 
we denote their root by $r$. Along that induction we shall prove the 
following property: if a vertex-end lives in a part $P_t$ of $(T,\Vcal)$, then it is dominated by 
$k+1$ vertices contained in the separator $V(P_t)\cap V(P_s)$, where $s$ is the neighbour of $t$ in 
the tree $T$ that is nearer to the root. (In particular, no vertex-end lives in the root part 
$P_r$, which may be assumed to consist of finitely many vertices). 

We remark that in our proof of \autoref{undom_td} we could construct the tree-decomposition 
$(T,\Vcal)$ such that if a vertex-end lives in a part $P_t$, then it is dominated by at least one 
vertex  contained in the separator $V(P_t)\cap V(P_s)$ (with $s$ as above). Indeed, we just have to 
use the variant of \autoref{star_td}, where we replace `each topological end'
by `each vertex-end not dominated by a vertex of $W$'. This variant is deduced with the same proof 
from \autoref{nested} except that we replace `topological end' by `vertex-end not dominated by a 
vertex of the set $W$'.  
During this proof we 
assume that  \autoref{undom_td} includes this additional statement. 

So the base case $k=0$ follows from \autoref{undom_td}.

Now assume that we already have a suitable tree-decomposition $(T,\Vcal)=(T,P_t|t\in V(T))$ for $k$.
We take each torso of a part of that tree-decomposition and delete the separator $V(P_t)\cap 
V(P_s)$. Call the resulting graph $H_t$. 
Now we apply \autoref{undom_td} to the graph $H_t$. Call this tree-decomposition $(T[t],\Vcal[t])$. 
We obtain a tree-decomposition of the part $P_t$ by adding the separator $V(P_t)\cap 
V(P_s)$ to all parts of the tree-decomposition $(T[t],\Vcal[t])$. Call that tree-decomposition 
$(T'[t],\Vcal'[t])$. 
 
We obtain a suitable tree-decomposition from  $(T,\Vcal)$ by replacing each part $P_t$ by the 
tree-decomposition $(T'[t],\Vcal'[t])$. This is well-defined as each separator of $(T,\Vcal)$ is a 
complete subgraph of $H_t$. Indeed, then we can attach the parts of $(T,\Vcal)$ above $P_t$ at the 
unique part of $(T'[t],\Vcal'[t])$ nearest to the root that includes the corresponding separator. 
It is straightforward to check that this tree-decomposition has the desired property. 
\end{proof}

In this paper we considered various classes of vertex-ends.
In \autoref{class_inclusion}, we depict the inclusion-relations that hold between the classes of 
vertex-ends considered in this paper.
  In \autoref{tabelle} we summarise which classes of 
vertex-ends are tree-distinguishable and which classes have end-faithful spanning trees.
We recall that countable graphs have normal spanning trees and hence end-faithful spanning trees 
for vertex-ends. Hence the questions of this paper are of particular interest for uncountable 
graphs. 
\begin{figure}
 \begin{center}
   	  \includegraphics[height=5.5cm]{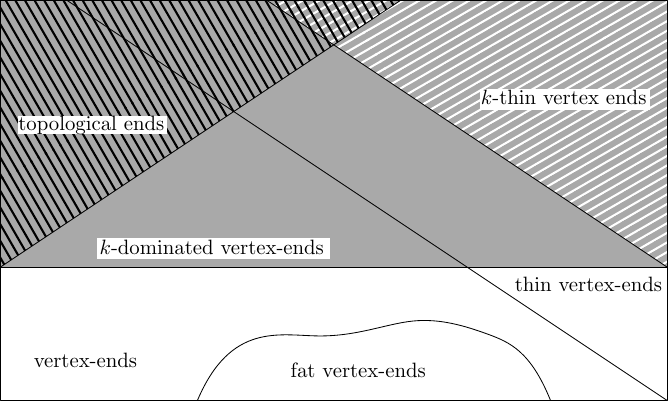}
   	  \caption{The classes of vertex-ends considered in this paper. }\label{class_inclusion}
\end{center}
   \end{figure}

\begin{figure}
 \begin{center}
   	  \includegraphics[height=6cm]{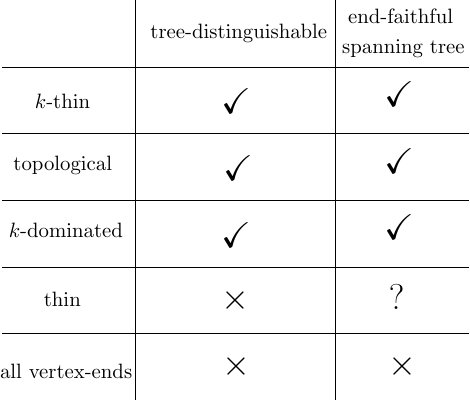}
   	  \caption{We put a tick in an entry of this table if the corresponding class of 
vertex-ends is tree-distinguishable or there is an end-faithful spanning tree for that class, 
respectively. If it is false we put a cross. In the one case where it is open, we 
put a question mark.}\label{tabelle}
\end{center}
   \end{figure}

All positive results of \autoref{tabelle} are proved in this paper. The counterexamples 
corresponding to the cross in the bottom right corner were constructed by Seymour and Thomas and 
Thomassen as mentioned in the introduction. The other two crosses are derived from 
\autoref{sketch_t2} and \autoref{thin_not_distinguished}.   
The question mark in \autoref{tabelle} corresponds to the following question. 
   
\begin{que}\label{q_thin}
Does every graph have an end-faithful spanning tree for the thin vertex-ends? 
\end{que}

The strengthening of \autoref{q_thin} with `thin' replaced by the class of vertex-ends that are 
dominated by finitely (or more generally: countably) many vertices is also open. Since this class 
contains the topological ends, 
this possible strengthening implies \autoref{Halin_type_intro}. 

However, by \autoref{thin_not_distinguished} the class of thin vertex-ends is not 
tree-distinguishable. Thus, if the answer to \autoref{q_thin} was `yes', the obvious strategy 
suggested by \autoref{proof_more_general}  directly via tree-decompositions defining precisely the 
thin ends cannot succeed. 

In the introduction of this paper we claimed that we repaired Halin's Conjecture. In the presence 
of \autoref{q_thin} this might deserve some further justification. 
Given the counterexamples against Halin's original conjecture, we explain the subtle difference 
between the following two questions.
\begin{enumerate}
 \item How can Halin's Conjecture be repaired?
 \item What is the largest possible natural subclass of the vertex-ends for which the weakening of 
Halin's Conjecture is true? 
\end{enumerate}
The second question is still open. Candidates for that subclass are the $k$-dominated vertex-ends, 
or more generally the finitely or countably dominated ones. 
These subclasses are vertex-ends with some finiteness or countability assumption. 
Unlike for the topological ends, 
it would not have been natural to ask the weakening of Halin's Conjecture for those subclasses in 
1964. Hence these subclasses can hardly 
give answers to question one.
The thin vertex-ends are also vertex-ends together with some finiteness assumption and might have 
served as a solution to question one if Diestel's original problem was true. Since the original 
problem is not true (and in fact can be repaired for the topological ends), the author is convinced 
that \autoref{Halin_type_intro} is the most natural way to repair Halin's Conjecture.

\section*{Acknowledgement}

I thank two anonymous referees whose valuable comments improved this paper. Furthermore I thank 
Nathan Bowler for pointing out an error in an earlier version. 

\bibliographystyle{plain}
\bibliography{literatur}

\end{document}